\documentclass{article}%
\usepackage{graphicx}
\usepackage{amsmath}
\usepackage{amsfonts}
\usepackage{amssymb}%
\providecommand{\U}[1]{\protect\rule{.1in}{.1in}}
\providecommand{\U}[1]{\protect\rule{.1in}{.1in}}
\providecommand{\U}[1]{\protect\rule{.1in}{.1in}}
\providecommand{\U}[1]{\protect\rule{.1in}{.1in}}
\providecommand{\U}[1]{\protect\rule{.1in}{.1in}}
\providecommand{\U}[1]{\protect\rule{.1in}{.1in}}
\providecommand{\U}[1]{\protect\rule{.1in}{.1in}}
\providecommand{\U}[1]{\protect\rule{.1in}{.1in}}
\providecommand{\U}[1]{\protect\rule{.1in}{.1in}}
\providecommand{\U}[1]{\protect\rule{.1in}{.1in}}
\providecommand{\U}[1]{\protect\rule{.1in}{.1in}}
\providecommand{\U}[1]{\protect\rule{.1in}{.1in}}
\providecommand{\U}[1]{\protect\rule{.1in}{.1in}}
\providecommand{\U}[1]{\protect\rule{.1in}{.1in}}
\providecommand{\U}[1]{\protect\rule{.1in}{.1in}}
\providecommand{\U}[1]{\protect\rule{.1in}{.1in}}
\providecommand{\U}[1]{\protect\rule{.1in}{.1in}}
\providecommand{\U}[1]{\protect\rule{.1in}{.1in}}
\providecommand{\U}[1]{\protect\rule{.1in}{.1in}}
\providecommand{\U}[1]{\protect\rule{.1in}{.1in}}
\providecommand{\U}[1]{\protect\rule{.1in}{.1in}}
\providecommand{\U}[1]{\protect\rule{.1in}{.1in}}
\providecommand{\U}[1]{\protect\rule{.1in}{.1in}}
\providecommand{\U}[1]{\protect\rule{.1in}{.1in}}
\providecommand{\U}[1]{\protect\rule{.1in}{.1in}}
\providecommand{\U}[1]{\protect\rule{.1in}{.1in}}
\providecommand{\U}[1]{\protect\rule{.1in}{.1in}}
\providecommand{\U}[1]{\protect\rule{.1in}{.1in}}
\providecommand{\U}[1]{\protect\rule{.1in}{.1in}}

\newtheorem{theorem}{Theorem}
{}

\newtheorem{conclusion}{Conclusion}

\newtheorem{corollary}{Corollary}

\newtheorem{definition}{Definition}

\newtheorem{lemma}{Lemma}
{}
\newtheorem{notation}{Notation}

\newtheorem{proposition}{Proposition}
\newtheorem{remark}{Remark}

\newenvironment{proof}[1][Proof]{\textbf{#1.} }{\ \rule{0.5em}{0.5em}}

\begin{document}

\title{Essential Spectral Singularities and the Spectral Expansion for the Hill Operator}
\author{O. A. Veliev\\{\small \ Depart. of Math., Dogus University, }\\{\small Ac\i badem, 34722, Kadik\"{o}y, \ Istanbul, Turkey.}\\\ {\small e-mail: oveliev@dogus.edu.tr}}
\date{}
\maketitle

\begin{abstract}
In this paper we investigate the spectral expansion for the one-dimensional
Schrodinger operator with a periodic complex-valued potential. For this we
consider in detail the spectral singularities and introduce new concepts as
essential spectral singularities and singular quasimomenta.



\end{abstract}

\section{ Introduction and Preliminary Facts}

In this paper we investigate the one dimensional Schr\"{o}dinger operator
$L(q)$ generated in $L_{2}(-\infty,\infty)$ by the differential expression
\begin{equation}
l(y)=-y^{^{\prime\prime}}(x)+q(x)y(x),
\end{equation}
where $q$ is $1$-periodic, Lebesgue integrable on $[0,1]$ and complex-valued
potential. Without loss of generality, we assume that the integral of $q$ over
$[0,1]$ is $0.$ It is well-known [1, 7, 8] that the spectrum $\sigma(L)$ of
the operator $L$ is the union of the spectra $\sigma(L_{t})$ of the operators
$L_{t}(q)$ for $t\in(-\pi,\pi]$ generated in $L_{2}[0,1]$ by (1) and the
boundary conditions
\begin{equation}
y(1)=e^{it}y(0),\text{ }y^{^{\prime}}(1)=e^{it}y^{^{\prime}}(0).
\end{equation}
The eigenvalues of $L_{t}$ are the roots of the characteristic equation
\begin{equation}
\Delta(\lambda,t)=:\left\vert
\begin{array}
[c]{cc}%
\theta(1,\lambda)-e^{it} & \varphi(1,\lambda)\\
\text{ }\theta^{\prime}(1,\lambda) & \varphi^{\prime}(1,\lambda)-e^{it}%
\end{array}
\right\vert =0
\end{equation}
of the operator $L_{t}$ which equivalent to
\begin{equation}
F(\lambda)=2\cos t,
\end{equation}
where $F(\lambda)=:\varphi^{\prime}(1,\lambda)+\theta(1,\lambda)$ is the Hill
discriminant, $\theta(x,\lambda)$ and $\varphi(x,\lambda)$ are the solutions
of the equation $l(y)=\lambda y$ satisfying the following initial conditions%

\begin{equation}
\theta(0,\lambda)=\varphi^{\prime}(0,\lambda)=1,\text{ }\theta^{\prime
}(0,\lambda)=\varphi(0,\lambda)=0.
\end{equation}

We consider the spectral expansion of the non-self-adjoint operator $L(q)$.
The spectral expansion for the self-adjoint operator $L(q)$\ was constructed
by Gelfand [3] and Titchmarsh [11]. The existence of the spectral
singularities and the absence of the Parseval's equality for the
nonself-adjoint operator $L_{t}$ do not allow us to apply the elegant method
of Gelfand (see\ [3]) for the construction of the spectral expansion for the
nonself-adjoin operator $L(q)$. Note that the spectral singularities of the
operator $L(q)$ are the points of its spectrum in neighborhoods of which the
projections of $L(q)$ are not uniformly bounded (see [15] and [4]). McGarvey
[7, 8] proved that $L(q)$ is a spectral operator if and only if the
projections of the operators $L_{t}(q)$ are bounded uniformly with respect to
$t$ in $(-\pi,\pi]$. Tkachenko [12] proved that the non-self-adjoint operator
$L$ can be reduced to triangular form if all eigenvalues of the operators
$L_{t}$ for $t\in(-\pi,\pi]$ are simple. However, in general, the eigenvalues
are not simple and the projections of the operators $L_{t}$ are not uniformly
bounded. Indeed, Gasymov's paper [2] shows that the operators $L_{t}(q)$ with
the potential $q$ of the form \
\begin{equation}
q(x)=%
{\textstyle\sum\limits_{n=1}^{\infty}}
q_{n}e^{inx},\text{ }%
{\textstyle\sum\limits_{n}}
\mid q_{n}\mid<\infty
\end{equation}
have infinitely many multiple eigenvalues, their projections are not uniformly
bounded and no one operator $L(q)$ with nonzero potential of type (6) is
spectral, since they have, in general, infinitely many spectral singularities.
Gasymov in [2] investigated the direct and inverse problems of the operator
$L(q)$ with potential (6) and derived a regularized spectral expansion. The
method of [2] is applicable only for the potentials of type (6). Gesztezy and
Tkachenko [4] proved two versions of a criterion for the operator $L(q)$ with
$q\in L_{2}[0,1]$ to be a spectral operator of scalar type, in sense of
Dunford, one analytic and one geometric. The analytic version was stated in
term of the solutions of Hill equation. The geometric version of the criterion
uses algebraic and geometric \ properties of the spectra of
periodic/antiperiodic and Dirichlet boundary value problems.

The problem of describing explicitly, for which potentials $q$ the Hill
operators $L(q)$ are spectral operators appears to have been open for about 50
years. Moreover, the discussed papers show that the set of potentials $q$ for
which $L(q)$ is spectral is a small subset of the periodic functions and it is
very hard to describe explicitly the required subset. In paper [19] we found
the explicit conditions on the potential $q$ such that $L(q)$ is an
asymptotically spectral operator and in [20] we constructed the spectral
expansion for the asymptotically spectral operator. However, the set of the
potentials constructed in [19] is also a small subset of the periodic
functions. Thus the theory of spectral operators is ineffective for the
construction of the spectral expansion for the nonself-adjoint periodic
differential operators. It is connected with the complicated picture of the
projections of the Hill operator with general complex potential. In this
paper, we construct the spectral expansion for the operator $L(q)$ with
arbitrary complex-valued locally integrable and periodic potential $q$. In
other word, we investigate in detail the spectral expansion for the general
and frequent case when the operator $L$ is not an asymptotically spectral and
hence is not a spectral operator. For this we introduce new concepts as
essential spectral singularities (ESS) and singular quasimomenta defined in
Definitions 3 and Definition 4.

To discuss more precisely the obtained results and to give the brief scheme of
this paper we need some preliminary facts about:

\textbf{(a) the eigenvalues of }$L_{t}(q)$\textbf{ and spectrum of }$L(q),$

\textbf{(b) the eigenfunction of }$L_{t}(q)$\textbf{,}

\textbf{(c) the spectral singularities of }$L(q)$\textbf{,}

\textbf{(d) the problems of the spectral expansion of }$L.$

Note that there are a lot of papers about the spectra of $L_{t}$ and $L$ (see
[1, 4, 18] and references on them). Here we introduce the facts which are used
essentially for the construction of the spectral expansion. In this section,
after introducing the preliminary facts we give the definition of ESS and
discuss its importance in the construction of the spectral expansion. In
section 2 we investigate the spectral singularities and ESS. In Section 3 we
construct spectral expansion for the operator $L(q)$ in term of the improper
integrals by using the ESS, singular quasimomenta and some parenthesis.
Finally, we explain (see Conclusion 1) why it is necessary to use the improper
integrals and parenthesis.

\textbf{(a) On the eigenvalues of }$L_{t}(q)$\textbf{ and spectrum of }$L(q).$

In the case $q=0$ the eigenvalues and eigenfunctions of $L_{t}(q)$ are $(2\pi
n+t)^{2}$ and $e^{i(2\pi n+t)x}$ for $n\in\mathbb{Z}$ respectively. In [17] we
proved that the large eigenvalues of the operators $L_{t}(q)$\ for $t\neq
0,\pi$ consist of the sequence $\left\{  \lambda_{n}(t):\mid n\mid
\gg1\right\}  $ satisfying\ \
\begin{equation}
\lambda_{n}(t)=(2\pi n+t)^{2}+O(n^{-1}\ln\left\vert n\right\vert )
\end{equation}
as $n\rightarrow\infty$ and the formula (7) is uniform with respect to $t$ in
$Q_{h},$ where
\begin{equation}
Q_{h}=\{t\in Q,\text{ }|t-\pi k|\geq h,\text{ }k=0,\pm1\},
\end{equation}
$h\in(0,1)$ and
\begin{equation}
Q=\{z\in\mathbb{C}:|\operatorname{Im}z|<1,-\pi<\operatorname{Re}z<\pi+1\}.
\end{equation}
Note that, the formula\ $f(n,t)=O(g(n))$ as $n\rightarrow\infty$ is said to be
uniform with respect to $t$ in a set $I$ if there exist positive constants $M$
and $N,$ independent of $t,$ such that $\mid f(n,t))\mid<M\mid g(n)\mid$ for
all $t\in I$ and $\mid n\mid\geq N.$ Moreover, it follows from (7) that for
any fixed $h$ ($h\in(0,1)$) there exists an integer $N(h)$ and positive
constant $M(h)$ such that for $|n|>N(h)$ and $t\in Q_{h}$ there exists unique
eigenvalue $\lambda_{n}(t),$ counting multiplicity, satisfying
\begin{equation}
\left\vert \lambda_{n}(t)-(2\pi n+t)^{2}\right\vert \leq n^{-1}M(h).
\end{equation}
Thus $\lambda_{n}(t)$ is simple for all $|n|>N(h)$ and $t\in Q_{h}.$

Besides, as it was shown in [19, 20], the integer $N(h)$ can be chosen so that
for $\left\vert t\right\vert \leq h$ and $|n|>N(h)$ there exist two
eigenvalues, counting multiplicity, denoted by $\lambda_{n}(t)$ and
$\lambda_{-n}(t)$ and satisfying
\begin{equation}
\left\vert \lambda_{\pm n}(t)-(2\pi n+t)^{2}\right\vert \leq15\pi nh.
\end{equation}
Similarly, for $\left\vert t-\pi\right\vert \leq h$ and $|n|>N(h)$ there exist
two eigenvalues, counting multiplicity, denoted by $\lambda_{n}(t)$ and
$\lambda_{-(n+1)}(t)$ such that%

\begin{equation}
\left\vert \lambda_{n}(t)-(2\pi n+t)^{2}\right\vert \leq15\pi nh,\text{
}\left\vert \lambda_{-(n+1)}(t)-(2\pi n+t)^{2}\right\vert \leq15\pi nh.
\end{equation}

As we noted above the spectrum $\sigma(L(q))$ of $L(q)$ is the union of the
eigenvalues of $L_{t}$ for all $t\in(-\pi,\pi].$ In [20] we proved that the
eigenvalues of $L_{t}$ can be numbered (counting the multiplicity) by elements
of $\mathbb{Z}$ such that, for each $n$ the function $\lambda_{n}(t)$ is
continuous on $[0,\pi]$ and for $|n|>N(h)$ the inequalities (10)-(12) hold.
The eigenvalues of $L_{-t}(q)$ coincides with the eigenvalues of $L_{t}(q),$
because they are roots of equation (4) and $\cos(-t)=\cos t.$ We define the
eigenvalue $\lambda_{n}(-t)$ of $L_{-t}(q)$ by $\lambda_{n}(-t)=\lambda
_{n}(t)$ for all $t\in(0,\pi).$ Thus
\begin{equation}
\sigma(L(q))=%
{\textstyle\bigcup\limits_{n\in\mathbb{Z}}}
\Gamma_{n},
\end{equation}
where
\begin{equation}
\Gamma_{n}=\left\{  \lambda_{n}(t):t\in\lbrack0,\pi]\right\}
\end{equation}
is a continuous curve.

The multiple eigenvalues of $L_{t}$ are the common roots of (4) and
$F^{^{\prime}}(\lambda)=0.$ Since the Hill discriminant $F(\lambda)$ is a
nonzero entire function, the set of zeros of $F^{^{\prime}}(\lambda)$ is at
most countable and can have no finite limit point. Let $\mu_{1},\mu_{2},...,$
be the roots of $F^{^{\prime}}(\lambda)=0$ and%

\begin{equation}
A=\{\pm t_{k}:k=1,2,...,\},\text{ }A_{n}=\left\{  \pm t_{k}:\mu_{k}\in
\Gamma_{n}\right\}  ,
\end{equation}
where $t_{k}=\arccos\tfrac{1}{2}F(\mu_{k}).$ Here the real part of the range
of usual principal value of $\arccos z$ is in $[0,\pi]$ and the imaginary part
is nonnegative and the quasimomenta $t+2\pi n$ for $n\in\mathbb{Z}$ also are
denoted by $t.$ In these notations and by (4) we have
\begin{equation}
F^{^{\prime}}(\mu_{k})=0,\text{ }\mu_{k}\in\sigma(L_{t_{k}})=\sigma(L_{-t_{k}%
}),\text{ }\mu_{k}\notin\sigma(L_{t}),\text{ }\forall t\neq\pm t_{k}.
\end{equation}
Thus, if $t\notin A,$ then all eigenvalues of $L_{t}$ are simple eigenvalues.
It follows from\textit{ }the well known asymptotic formulas for $F(\lambda)$
that (see [6]) the accumulation points of the set $A\cap Q$ are $0$ and $\pi.$
Therefore $\overline{A}=A\cup\{0,\pi\}.$

Suppose that the multiplicity of the eigenvalue $\mu_{k}$ is $j.$ If $t_{k}%
\in(-\pi,\pi]$ then $j$ components (14) of the spectrum of $L(q)$ meet at the
point$\mu_{k}$. If $\mu_{k}$ is a large number then $j\leq2.$ Therefore if
$\lambda_{n}(0)$ for large $n$ is the double eigenvalue of $L_{0}$, then it
readily follows from the numerations of the eigenvalues and (11) that the
components $\Gamma_{n}$ and $\Gamma_{-n}$ are joined. Similarly if
$\lambda_{n}(\pi)$ for large $n$ is the double eigenvalue of $L_{\pi}$, then
by (12) the components $\Gamma_{n}$ and $\Gamma_{-(n+1)}$ are joined.

\textbf{(b) On the eigenfunctions of }$L_{t}.$ In [17] we proved that the
normalized eigenfunction $\Psi_{n,t}(x)$ corresponding to the eigenvalue
$\lambda_{n}(t)$\ satisfies
\begin{equation}
\Psi_{n,t}(x)=\frac{1}{\parallel e^{itx}\parallel}e^{i(2n\pi+t)x}%
+h_{n,t},\text{ }\left\Vert h_{n,t}\right\Vert =O(n^{-1})
\end{equation}
and the formula (17) is uniform with respect to $t$ in $Q_{h}.$

Let $\Psi_{n,t}^{\ast}$ be the normalized eigenfunction of $(L_{t}(q))^{\ast}$
corresponding to $\overline{\lambda_{n}(t)}.$ The boundary condition adjoint
to (2) is
\begin{equation}
y(1)=e^{i\overline{t}}y(0),\text{ }y^{^{\prime}}(1)=e^{i\overline{t}%
}y^{^{\prime}}(0).
\end{equation}
Therefore, $(L_{t}(q))^{\ast}=$ $L_{\overline{t}}(\overline{q})$ and by (17),
we have the following uniform with respect to $t$ in $Q_{h}$ asymptotic
formula
\begin{equation}
\Psi_{n,t}^{\ast}(x)=\frac{1}{\parallel e^{i\overline{t}x}\parallel}e^{i(2\pi
n+\overline{t})x}+h_{n,t}^{\ast}(x),\text{ }\left\Vert h_{n,t}^{\ast
}\right\Vert =O(n^{-1}).
\end{equation}

Replacing first and second row of the characteristic determinant in (3) by the
row vector $(\theta(x,\lambda),$ $\varphi(x,\lambda))$ we obtain the functions%
\begin{equation}
G_{t}(x,\lambda)=\theta^{\prime}\varphi(x,\lambda)+(e^{it}-\varphi^{\prime
})\theta(x,\lambda)
\end{equation}
and
\begin{equation}
\Phi_{t}(x,\lambda)=\varphi\theta(x,\lambda)+(e^{it}-\theta)\varphi(x,\lambda)
\end{equation}
which for $\lambda=\lambda_{n}(t)$ are the eigenfunctions (if they are not the
zero functions) of $L_{t}$ corresponding to the eigenvalue $\lambda_{n}(t)$,
where for simplicity of the notations $\varphi(1,\lambda),$ $\theta
(1,\lambda),$ $\varphi^{\prime}(1,\lambda)$ and $\theta^{\prime}(1,\lambda)$
are denoted by $\varphi,\theta,\varphi^{\prime}$ and $\theta^{\prime}$
respectively. Then the normalized eigenfunctions $\Psi_{n,t}(x)$ and
$\Psi_{n,t}^{\ast}(x)$ for $t\in(-\pi,\pi]$ can be written in the form
\begin{equation}
\Psi_{n,t}(x)=\frac{\Phi_{t}(x,\lambda_{n}(t))}{\left\Vert \Phi_{t}%
(\cdot,\lambda_{n}(t))\right\Vert },\text{ }\Psi_{n,t}^{\ast}=\frac
{\overline{\Phi_{-t}(x,\lambda_{n}(t))}}{\left\Vert \Phi_{-t}(\cdot
,\lambda_{n}(t))\right\Vert }%
\end{equation}
or
\begin{equation}
\Psi_{n,t}(x)=\frac{G_{t}(x,\lambda_{n}(t))}{\left\Vert G_{t}(\cdot
,\lambda_{n}(t))\right\Vert },\text{ }\Psi_{n,t}^{\ast}=\frac{\overline
{G_{-t}(x,\lambda_{n}(t))}}{\left\Vert G_{-t}(\cdot,\lambda_{n}(t))\right\Vert
}.
\end{equation}
It is well known that [9, 10] for each $t\notin\overline{A}$ the system
$\{\Psi_{n,t}:n\in\mathbb{Z}\}$ is a Reisz basis of $L_{2}[0,1]$ and $\left\{
X_{n,t}:n\in\mathbb{Z}\right\}  ,$ defined by
\begin{equation}
X_{n,t}=\frac{1}{\overline{\alpha_{n}(t)}}\Psi_{n,t}^{\ast},\text{ }\alpha
_{n}(t)=(\Psi_{n,t},\Psi_{n,t}^{\ast}),
\end{equation}
is the beorthogonal system, where $(\cdot,\cdot)$ is the inner product in
$L_{2}[0,1]$.

\textbf{(c) On the} \textbf{spectral singularities of} $L.$ Since the spectral
singularities of the operator $L(q)$ are the points of its spectrum in
neighborhoods of which the projections of $L(q)$ are not uniformly bounded, to
consider the spectral singularities, first we need to discuss the projections
of the operators $L_{t}(q)$ and $L(q)$. It is well-known that (see p. 39 of
[10]) if $\lambda_{n}(t)$ is a simple eigenvalue of $L_{t},$ then the spectral
projection $e(t,\gamma)$ defined by contour integration of the resolvent of
$L_{t}(q)$, where $\gamma$ is the closed contour containing only the
eigenvalue $\lambda_{n}(t),$ has the form%
\begin{equation}
e(t,\gamma)f=\frac{1}{\alpha_{n}(t)}(f,\Psi_{n,t}^{\ast})\Psi_{n,t}.\text{ }%
\end{equation}
One can easily verify that.
\begin{equation}
\left\Vert e(t,\gamma)\right\Vert =\left\vert \frac{1}{\alpha_{n}%
(t)}\right\vert .
\end{equation}

In [15] we defined projection $P(\gamma)$ of $L$ for the arc $\gamma
\subset\Gamma_{n}$ which does not contain the multiple eigenvalues of the
operators $L_{t},$ as follows
\begin{equation}
P(\gamma)=\lim_{\varepsilon\rightarrow0}\int\limits_{\gamma_{\varepsilon}%
^{1}\cup\gamma_{\varepsilon}^{2}}(L-\lambda I)^{-1}d\lambda,
\end{equation}
where $\gamma_{\varepsilon}^{1}\subset\rho(L)$ and $\gamma_{\varepsilon}%
^{2}\subset\rho(L)$ are the connected curves lying in opposite sides of
$\gamma$ and
\[
\lim_{\varepsilon\rightarrow0}\gamma_{\varepsilon}^{i}=\gamma,\text{ }%
\forall=i=1,2.
\]
Here $\rho(L)$ denotes the resolvent set of $L.$ Moreover, we proved that if
additionally the derivative of the characteristic determinant with respect to
the quasimomentum $t$ is nonzero, which equivalent to the condition
$\lambda_{n}^{^{\prime}}(t)\neq0$ and holds for $\lambda_{n}(t)\in\gamma,$
$t\neq0,\pi,$ then
\begin{equation}
P(\gamma)f=\frac{1}{2\pi}%
{\textstyle\int\limits_{\delta}}
\frac{1}{\alpha_{n}(t)}(f,\Psi_{n,t}^{\ast})_{\mathbb{R}}\Psi_{n,t}dt,\text{
}\left\Vert P(\gamma)\right\Vert =\sup_{t\in\delta}\frac{1}{\left\vert
\alpha_{n}(t)\right\vert },
\end{equation}
where $\delta=\left\{  t\in(-\pi,\pi]:\lambda_{n}(t)\in\gamma\right\}  $ and
$(\cdot,\cdot)_{I}$ for any set $I$ denotes the inner product in $L_{2}(I).$
Thus the uniform boundedness of the projections $P(\gamma)$ and hence the
existence of the spectral singularities depend on the behavior of $\alpha
_{n}(t).$ To investigate $\alpha_{n}(t)$ we use the formula
\begin{equation}
\alpha_{n}(t)=-\frac{\varphi(1,\lambda_{n}(t))F^{^{\prime}}(\lambda_{n}%
(t))}{\left\Vert \Phi_{t}(\cdot,\lambda_{n}(t))\right\Vert \left\Vert
\Phi_{-t}(\cdot,\lambda_{n}(t))\right\Vert }%
\end{equation}
which immediately follows from (24), (22), (21), (4), the Wronskian equality
\begin{equation}
\theta\varphi^{\prime}-\theta^{\prime}\varphi=1
\end{equation}
and the formula
\begin{equation}
F^{^{\prime}}(\lambda)=%
{\textstyle\int\limits_{0}^{1}}
\theta^{\prime}\varphi^{2}(x,\lambda)+\left(  \theta-\varphi^{\prime}\right)
\theta(x,\lambda)\varphi(x,\lambda)-\varphi\theta^{2}(x,\lambda)dx
\end{equation}
obtained in [11] (see (21.4.5) in Section 21 of [11]). Instead of (22) and
(21) using (23) and (20) we obtain
\begin{equation}
\alpha_{n}(t)=-\frac{\theta^{\prime}(1,\lambda_{n}(t))F^{^{\prime}}%
(\lambda_{n}(t))}{\left\Vert G_{t}(\cdot,\lambda_{n}(t))\right\Vert \left\Vert
G_{-t}(\cdot,\lambda_{n}(t))\right\Vert }.
\end{equation}

In [4], the projections were defined as follows. By Definition 2.4 of [4], a
closed arc $\gamma=:\{z\in\mathbb{C}:z=\lambda(t),t\in\lbrack\alpha,\beta]\}$
with $\lambda(t)$ analytic in an open neighborhood of $[\alpha,\beta]$ and
\[
F(\lambda(t))=2\cos t,\text{ }F^{^{\prime}}(\lambda(t))\neq0,\text{ }\forall
t\in\lbrack\alpha,\beta],\text{ }\lambda^{^{\prime}}(t)\neq0,\text{ }\forall
t\in(\alpha,\beta)
\]
is called a regular spectral arc of $L(q).$ The projection $\widetilde
{P}(\gamma)$ corresponding to the regular spectral arc $\gamma$ was defined
by
\begin{equation}
\widetilde{P}(\gamma)=\frac{1}{2\pi}%
{\textstyle\int\limits_{\gamma}}
(\Phi_{+}(x,\lambda)F_{-}(\lambda,f)+\Phi_{-}(x,\lambda)F_{+}(\lambda
,f))\frac{1}{\varphi p(\lambda)}d\lambda,
\end{equation}
where
\begin{equation}
\Phi_{\pm}(x,\lambda)=\varphi\theta(x,\lambda)+\tfrac{1}{2}(\varphi^{\prime
}-\theta\pm ip(\lambda))\varphi(x,\lambda),
\end{equation}%
\begin{equation}
\text{ }p(\lambda)=\sqrt{4-F^{2}(\lambda)},\text{ }F_{\pm}(\lambda
,f)=\int_{\mathbb{R}}f(x)\Phi_{\pm}(x,\lambda)dx.
\end{equation}
Using (4) one can readily see that
\begin{equation}
\Phi_{\pm}(x,\lambda_{n}(t))=\Phi_{\pm t}(x,\lambda_{n}(t)),
\end{equation}
where $\Phi_{\pm t}$ is defined in (21). If $\gamma\subset\Gamma_{n}$ then
changing the variable $t$ to the variable $\lambda$ in the integral (28),
using the formulas (29), (22), (21) and (36) and taking into account the
equalities $\lambda_{n}(-t)=\lambda_{n}(t),$ $\frac{dt}{d\lambda}%
=-\frac{F^{^{\prime}}(\lambda)}{p(\lambda)}$ which follows from (4) we obtain%
\begin{equation}%
{\textstyle\int\limits_{\delta}}
\frac{1}{\alpha_{n}(t)}(f,\Psi_{n,t}^{\ast})_{\mathbb{R}}\Psi_{n,t}(x)dt=%
{\textstyle\int\limits_{\gamma}}
(\Phi_{+}(x,\lambda)F_{-}(\lambda,f)+\Phi_{-}(x,\lambda)F_{+}(\lambda
,f))\frac{1}{\varphi p(\lambda)}d\lambda.
\end{equation}
Therefore, by (28) and (33) we have $P(\gamma)=\widetilde{P}(\gamma).$ Hence
in the both cases the projection of $L(q)$ and its norm are defined by (28).
Moreover, one can readily see that the curve $\gamma$ used in (28) is the same
with the regular spectral arc defined in [4]. Thus in [15] and [4] the
spectral singularities was defined as follows.

\begin{definition}
We say that $\lambda\in\sigma(L(q))$ is a spectral singularity of $L(q)$ if
for all $\varepsilon>0$\ there exists a sequence $\{\gamma_{n}\}$ of the
regular spectral arcs $\gamma_{n}\subset\{z\in\mathbb{C}:\mid z-\lambda
\mid<\varepsilon\}$ such that
\begin{equation}
\lim_{n\rightarrow\infty}\parallel P(\gamma_{n})\parallel=\infty.
\end{equation}

\end{definition}

In the similar way, we defined in [19] the spectral singularity at infinity.

\begin{definition}
We say that the operator $L$ has a spectral singularity at infinity if there
exists a sequence $\{\gamma_{n}\}$ of the regular spectral arcs such that
$d(0,\gamma_{n})\rightarrow\infty$ as $n\rightarrow\infty$ and (38) holds,
where $d(0,\gamma_{n})$ is the distance from the point $(0,0)$ to the arc
$\gamma_{n}.$
\end{definition}

The following proposition follows immediately from (28) and the definitions 1
and 2

\begin{proposition}
$(a)$ $\lambda\in\sigma(L(q))$ is a spectral singularity of $L(q)$ if and only
if there exist $n\in\mathbb{Z}$ and sequence $\{t_{k}\}\subset(-\pi
,\pi]\backslash\overline{A}$ such that $\lambda_{n}(t_{k})\rightarrow\lambda$
and $\alpha_{n}(t_{k})\rightarrow0$ as $k\rightarrow\infty.$

$(b)$ The operator $L$ has a spectral singularity at infinity if and only if
there exist sequences $\left\{  n_{k}\right\}  \in\mathbb{Z}$ and
$\{t_{k}\}\subset(-\pi,\pi]\backslash\overline{A}$ such that $\alpha_{n_{k}%
}(t_{k})\rightarrow0$ as $k\rightarrow\infty.$
\end{proposition}

Thus the spectral singularities and hence the spectrality of $L(q)$ is
connected with the uniform boundedness of $\frac{1}{\alpha_{n}},$ while as we
see below the spectral expansion of $L(q)$ is essentially connected with the
integrability of this function.

\textbf{(d) On the problems of the spectral expansion of} $L.$ By Gelfand's
Lemma (see [3]) for every $f\in L_{2}(-\infty,\infty)$ there exists $f_{t}(x)$
such that
\begin{equation}
f(x)=\frac{1}{2\pi}\int\limits_{0}^{2\pi}f_{t}(x)dt,
\end{equation}%
\begin{equation}
f_{t}(x)=\sum\limits_{k=-\infty}^{\infty}f(x+k)e^{ikt},\text{ }\int_{-\infty
}^{\infty}\left\vert f(x)\right\vert ^{2}dx=\frac{1}{2\pi}\int\limits_{0}%
^{2\pi}\int\limits_{0}^{1}\left\vert f_{t}(x)\right\vert ^{2}dxdt
\end{equation}
and%
\begin{equation}
f_{t}(x+1)=e^{it}f_{t}(x).
\end{equation}

Let $h\in(0,1)\backslash A,$ and let $l$\ be a continuous curve joining the
points $-\pi+h$ and $\pi+h$ and satisfying
\begin{equation}
l\subset Q_{h}\backslash A,\text{ }%
\end{equation}
where $Q_{h}$ and $A$ are defined in (8) and (15) respectively. If $f$ is a
compactly supposed and continuous function, then $f_{t}(x)$ is an analytic
function of $t$ in a neighborhood of $\overline{D}$ for each $x,$ where
$\overline{D}$ is the closure of the domain enclosed by $l\cup\lbrack
-\pi+h,\pi+h].$ Hence the Cauchy's theorem and (39), (41) give
\begin{equation}
f(x)=\frac{1}{2\pi}\int_{l}f_{t}(x)dt.
\end{equation}
On the other hand, for each $t\in l$ we have a decomposition
\begin{equation}
f_{t}(x)=\sum_{n\in\mathbb{Z}}a_{n}(t)\Psi_{n,t}(x)
\end{equation}
of $f_{t}(x)$ by the basis $\{\Psi_{n,t}:n\in\mathbb{Z}\},$ where
$a_{n}(t)=\int_{0}^{1}f_{t}(x)\overline{X_{n,t}(x)}dx=\tfrac{1}{\alpha_{n}%
(t)}(f_{t},\Psi_{n,t}^{\ast})$ (see \textbf{(b)}). Here $\Psi_{n,t}$ and
$X_{n,t}$ can be extended to $(-\infty,\infty)$ by%
\begin{equation}
\Psi_{n,t}(x+1)=e^{it}\Psi_{n,t}(x)\text{ }\And X_{n,t}(x+1)=e^{i\overline{t}%
}X_{n,t}(x).
\end{equation}
Then the following equality holds
\begin{equation}
\int_{0}^{1}f_{t}(x)\overline{X_{n,t}(x)}dx=\int_{-\infty}^{\infty
}f(x)\overline{X_{n,t}(x)}dx
\end{equation}
(see [3]). Using (44) in (43), we get
\begin{equation}
f(x)=\frac{1}{2\pi}\int\limits_{l}f_{t}(x)dt=\frac{1}{2\pi}\int\limits_{l}%
\sum_{n\in\mathbb{Z}}a_{n}(t)\Psi_{n,t}(x)dt.
\end{equation}

In [14, 16, 18] we proved that for the continuous curve $l\subset
Q_{h}\backslash A$ the series in (47) can be integrated term by term:
\begin{equation}
\int\limits_{l}\sum_{n\in\mathbb{Z}}a_{n}(t)\Psi_{n,t}(x)dt=\sum
_{n\in\mathbb{Z}}\int\limits_{l}a_{n}(t)\Psi_{n,t}(x)dt.
\end{equation}
Therefore we have
\begin{equation}
f(x)=\frac{1}{2\pi}\sum_{n\in\mathbb{Z}}\int\limits_{l}a_{n}(t)\Psi
_{n,t}(x)dt,
\end{equation}
where the series converges in the norm of $L_{2}(a,b)$ for every
$a,b\in\mathbb{R}.$

To get the spectral expansion in the term of $t$ from (49) we need to replace
the integrals over $l$ by the integral over $(-\pi,\pi]$. As we see in the
next section (see Lemma 1), expressions $a_{n}(t)\Psi_{n,t}$ and $\frac
{1}{\alpha_{n}(t)}$ are piecewise continuous on $(-\pi,\pi].$ If $\alpha
_{n}(t)\rightarrow0$ as $t\rightarrow c$ for some $c\in(-\pi,\pi]$ then
$\frac{1}{\alpha_{n}(t)}\rightarrow\infty$ and $a_{n}(t)\Psi_{n,t}%
(x)\rightarrow\infty$ for some $f$ and $x.$ By Proposition 1 the boundlessness
of $\frac{1}{\alpha_{n}}$ is the characterization of the spectral
singularities. Moreover, the considerations of the spectral singularities,
that is, the consideration of the boundlessness of $\frac{1}{\alpha_{n}}$ play
only the crucial rule for the investigations of the spectrality of $L$. On the
other hand, the papers [2, 4, 19] show that, in general, the Hill operator $L$
is not a spectral operator. Since $\frac{1}{\alpha_{n}}$ may have an
integrable boundlessness, its boundlessness is not a criterion for the
nonexistence of the integrals
\begin{equation}
\int\limits_{\delta}\frac{1}{\alpha_{n}(t)}(f,\Psi_{n,t}^{\ast})_{\mathbb{R}%
}\Psi_{n,t}(x)dt
\end{equation}
for $\delta\subset(-\pi,\pi].$ Hence to construct the spectral expansion for
the operator $L$ we need to introduce a new concept connected with the
existence of the integrals (50) for $\delta\subset(-\pi,\pi]$ which can be
reduced to the investigation of the integrability of $\frac{1}{\alpha_{n}}$
(see Remark 1 in Section 3). Therefore we introduce the following notions,
independent of the choice of $f,$ for the construction of the spectral expansion.

\begin{definition}
We say that a point $\lambda_{0}\in\sigma(L_{\pm t_{0}})\subset\sigma(L)$ is
an essential spectral singularity (ESS) of the operator $L$ if there exists
$n\in\mathbb{Z}$ such that $\lambda_{0}=\lambda_{n}(t_{0})$ and for each
$\varepsilon$ the function $\frac{1}{\alpha_{n}}$ is not integrable on
$\left(  (t_{0}-\varepsilon,t_{0}+\varepsilon)\cup(-t_{0}-\varepsilon
,-t_{0}+\varepsilon)\right)  \backslash A_{n}$.
\end{definition}

In this paper we investigate the spectral expansion by using the concept ESS.
First (in Section 2) we consider the concept ESS. Then, in Section 3, we
construct the spectral expansion for the Hill operator.

\section{Spectral Singularity and ESS}

By (24), (28) and the definitions 1 and 3 to consider the spectral
singularities and ESS we need to investigate the normalized eigenfunctions
$\Psi_{n,t}$ and $\Psi_{n,t}^{\ast}$ and then $\frac{1}{\alpha_{n}(t)}.$
Therefore, first we prove the following lemma.

\begin{lemma}
$(a)$ For each fixed $x$ the functions $\left\vert \Psi_{n,t}(x)\right\vert $,
$\left\vert \Psi_{n,t}^{\ast}(x)\right\vert $ and $\left\vert \alpha
_{n}\right\vert ,$ where $\Psi_{n,t}(x)$, $\Psi_{n,t}^{\ast}(x)$ and
$\alpha_{n}$ are defined in (22) and (24), are continuous functions at
$(-\pi,0)\cup(0,\pi)$.

$(b)$ If the geometric multiplicity of the eigenvalues $\lambda_{n}(0)$ and
$\lambda_{n}(\pi)$ is $1,$ then for each fixed $x$ the functions $\Psi
_{n,t}(x)$, $\Psi_{n,t}^{\ast}(x)$ and $\alpha_{n}$ are continuous at $0$ and
$\pi$\ respectively.

$(c)$ For each fixed $x,$ $\Psi_{n,t}(x)$ and $\Psi_{n,t}^{\ast}(x)$ are
bounded functions at $(-\pi,0)\cup(0,\pi)$.

$(d)$ The function $\frac{1}{\left\vert \alpha_{n}\right\vert }$ is continuous
in $\left(  (-\pi,0)\cup(0,\pi)\right)  \backslash A_{n},$ where $A_{n}$ is
defined in (15). For each fixed $x$ the functions $\Psi_{n,t}(x)$, $\Psi
_{n,t}^{\ast}(x)$ and $\alpha_{n}$ , $\frac{1}{\alpha_{n}}$ are piecewise
continuous at $(-\pi,0)\cup(0,\pi)$.
\end{lemma}

\begin{proof}
$(a)$ It is well-known that [1] for $t\in(-\pi,0)\cup(0,\pi),$ the operator
$L_{t}$ cannot have two linearly independent eigenfunctions corresponding to
one eigenvalue $\lambda_{n}(t)$. Indeed, otherwise, both solutions
$\varphi(x,\lambda_{n}(t))$ and $\theta(x,\lambda_{n}(t))$ satisfy the
boundary condition (2). But, it implies that%
\begin{equation}
\varphi^{\prime}(1,\lambda_{n}(t))=e^{it},\text{ }\theta(1,\lambda
_{n}(t))=e^{it}%
\end{equation}
which contradicts (4) for $t\neq0,\pi.$

As we noted\textbf{ }in introduction (see \textbf{(a))} $\lambda_{n}(t)$ is a
continuous function. Therefore it follows from (21) that $\Phi_{t}%
(x,\lambda_{n}(t))$ for each fixed $x$ depend continuously on $t.$ Moreover,
by the uniform boundedness theorem $\left\Vert \Phi_{t}(\cdot,\lambda
_{n}(t))\right\Vert $ continuously depend on $t.$ Since $\theta(x,\lambda)$
and $\varphi(x,\lambda)$ are linearly independent solution, it follows from
(21) that $\left\Vert \Phi_{t}(\cdot,\lambda_{n}(t))\right\Vert =0$ if and
only if $\varphi(\lambda_{n}(t))=0$ and $e^{it}-\theta(\lambda_{n}(t))=0,$
which is possible for at most finite number of $t.$ Thus there may exists a
finite set $B=\left\{  u_{1},\text{ }u_{2},...,u_{k}\right\}  \subset\left(
(-\pi,0)\cup(0,\pi)\right)  $ such that $\left\Vert \Phi_{t}(\cdot,\lambda
_{n}(t))\right\Vert =0$ for $t\in B.$ Hence
\begin{equation}
\frac{\Phi_{t}(x,\lambda_{n}(t))}{\left\Vert \Phi_{t}(\cdot,\lambda
_{n}(t))\right\Vert }%
\end{equation}
is continuous at $\left(  (-\pi,0)\cup(0,\pi)\right)  \backslash B.$ In the
same way we prove that there may exists a finite set $C=\left\{  v_{1},\text{
}v_{2},...,v_{m}\right\}  $ such that $\left\Vert G_{t}(\cdot,\lambda
_{n}(t))\right\Vert =0$ for $t\in C$ and
\begin{equation}
\frac{G_{t}(x,\lambda_{n}(t))}{\left\Vert G_{t}(\cdot,\lambda_{n}%
(t))\right\Vert }%
\end{equation}
is continuous at $\left(  (-\pi,0)\cup(0,\pi)\right)  \backslash C,$ where
$G_{t}$ is defined in (20). For

$t\in\left(  (-\pi,0)\cup(0,\pi)\right)  \backslash(C\cup B)$ the operator
$L_{t}$ has unique linearly independent eigenfunction. Hence there exists a
function $c(t)$ such that $\left\vert c(t)\right\vert =1$ and
\begin{equation}
\frac{\Phi_{t}(x,\lambda_{n}(t))}{\left\Vert \Phi_{t}(\cdot,\lambda
_{n}(t))\right\Vert }=c(t)\frac{G_{t}(x,\lambda_{n}(t))}{\left\Vert
G_{t}(\cdot,\lambda_{n}(t))\right\Vert },\text{ }\forall t\in\left(
(-\pi,0)\cup(0,\pi)\right)  \backslash(C\cup B).
\end{equation}
On the other hand if $t\in B\cap C,$ then (51) holds which is impossible for
$t\neq0,\pi.$ It means that $\left(  (-\pi,0)\cup(0,\pi)\right)  \cap(B\cap
C)$ is an empty set. Therefore, it follows from (54) that for each fixed $x$
the absolute value of (52), that is, $\left\vert \Psi_{n,t}(x)\right\vert $ is
continuous at $(-\pi,0)\cup(0,\pi).$ In the same way, the same statements can
be proved for $\Psi_{n,t}^{\ast}(x).$

To prove the continuity of $\left\vert \alpha_{n}\right\vert $ at
$(-\pi,0)\cup(0,\pi)$ we use (22)-(24). By (4) if

$e^{it}-\theta(\lambda_{n}(t))=0$ then $e^{-it}-\varphi^{\prime}(\lambda
_{n}(t))=0.$ Therefore
\[
\left(  \frac{\Phi_{t}(\cdot,\lambda_{n}(t))}{\left\Vert \Phi_{t}%
(\cdot,\lambda_{n}(t))\right\Vert },\frac{\overline{G_{-t}(\cdot,\lambda
_{n}(t))}}{\left\Vert G_{-t}(\cdot,\lambda_{n}(t))\right\Vert }\right)  \text{
}\And\text{ }\left(  \frac{G_{t}(\cdot,\lambda_{n}(t))}{\left\Vert G_{t}%
(\cdot,\lambda_{n}(t))\right\Vert },\frac{\overline{\Phi_{-t}(\cdot
,\lambda_{n}(t))}}{\left\Vert \Phi_{-t}(\cdot,\lambda_{n}(t))\right\Vert
}\right)
\]
are continuous at $\left(  (-\pi,0)\cup(0,\pi)\right)  \backslash B_{1}$ and
$\left(  (-\pi,0)\cup(0,\pi)\right)  \backslash C_{1}$ respectively, where
$B_{1}$ and $C_{1}$ are finite sets and $B_{1}\cap C_{1}=\varnothing.$ Thus
arguing as above, we see that $\left\vert \alpha_{n}\right\vert $ is
continuous at $(-\pi,0)\cup(0,\pi).$

$(b)$ Now suppose that the geometric multiplicity of $\lambda_{n}(0)$ is $1.$
Then at least one of the entry of characteristic determinant (see (3))
\begin{equation}
\left\vert
\begin{array}
[c]{cc}%
\theta(\lambda_{n}(0))-1 & \varphi(\lambda_{n}(0))\\
\text{ }\theta^{\prime}(\lambda_{n}(0)) & \varphi^{\prime}(\lambda
_{n}(0))-e^{it}%
\end{array}
\right\vert
\end{equation}
is not zero, that is, at least one of $\Phi_{0}(\cdot,\lambda_{n}(0))$ and
$G_{0}(\cdot,\lambda_{n}(0))$ is not zero function. Without loss of
generality, assume that $\Phi_{0}(\cdot,\lambda_{n}(0))$ is not zero function.
Then by (22) for each $x,$ $\Psi_{n,t}(x)$ and $\Psi_{n,t}^{\ast}(x)$ is
continuous at $0.$ The continuity of $\alpha_{n}$ follows from (24). In the
same way we prove that they are continuous at $\pi$ if the geometric
multiplicity of $\lambda_{n}(\pi)$ is $1.$

$(c)$ Now we prove that for each $x$ the function $\Psi_{n,t}(x)$ is bounded
at $(-\pi,0)\cup(0,\pi).$ By $(a)$ and $(b)$ it is enough to show that it is
bounded in some deleted neighborhoods of $0$ and $\pi,$ if the geometric
multiplicity of $\lambda_{n}(0)$ and $\lambda_{n}(\pi)$ is $2$ respectively.
We prove it for $t=0.$ The proof for $\ t=\pi$ is similar. If the geometric
multiplicity of $\lambda_{n}(0)$ is $2$ then all entries of (55) are zero and
hence $\varphi(\lambda_{n}(0))=0.$ Then it is clear that there exists
$\delta_{1}>0$ such that
\begin{equation}
\varphi(\lambda_{n}(t))\neq0
\end{equation}
for $0<\left\vert t\right\vert <\delta_{1}.$ Since $\theta(x,\lambda)$ and
$\varphi(x,\lambda)$ are continuous with respect to $(x,\lambda)$ and nonzero
functions and $\lambda_{n}(t)$ is continuous at $t=0,$ there exist constants
$M,$ $\varepsilon$ \ and $\delta_{2}$ such that%
\[
\left\vert \theta(x,\lambda_{n}(t))\right\vert <M,\text{ }\left\vert
\varphi(x,\lambda_{n}(t))\right\vert <M,\text{ }\left\Vert \varphi
(\cdot,\lambda_{n}(t))\right\Vert >\varepsilon,\text{ }\left\Vert \theta
(\cdot,\lambda_{n}(t))\right\Vert >\varepsilon
\]
for $x\in\lbrack0,1]$ and $\left\vert t\right\vert <\delta_{2}.$ On the other
hand, $\theta(x,\lambda)$ and $\varphi(x,\lambda)$ are linearly independent
solutions and hence they are linearly independent elements of $L_{2}(0,1)$
which implies that there exist positive constants $c<1$ and $\delta_{3}$ such
that%
\[
\left\vert \left(  \varphi(\cdot,\lambda_{n}(t)),\theta(\cdot,\lambda
_{n}(t))\right)  \right\vert <c\left\Vert \varphi(\cdot,\lambda_{n}%
(t))\right\Vert \text{ }\left\Vert \theta(\cdot,\lambda_{n}(t))\right\Vert
\]
for $\left\vert t\right\vert <\delta_{3}$. Using these inequalities one can
easily verify that there exist positive constants $M_{1},$ $\varepsilon$ \ and
$\delta$ such that
\begin{align*}
\left\vert \Phi_{t}(x,\lambda_{n}(t))\right\vert ^{2}  &  <M_{1}(\left\vert
\varphi(\lambda_{n}(t))\right\vert ^{2}+\left\vert e^{it}-\theta(\lambda
_{n}(t))\right\vert ^{2}),\\
\text{ }\left\Vert \Phi_{t}(\cdot,\lambda_{n}(t))\right\Vert ^{2}  &
>\varepsilon(\left\vert \varphi(\lambda_{n}(t))\right\vert ^{2}+\left\vert
e^{it}-\theta(\lambda_{n}(t)\right\vert ^{2})
\end{align*}
for $x\in\lbrack0,1]$ and $0<\left\vert t\right\vert <\delta.$ It with (56)
implies that $\Psi_{n,t}(x)$ is bounded in some deleted neighborhood of $0.$
In the same way we prove it for $\Psi_{n,t}^{\ast}(x).$

$(d)$ Since for $t\in\left(  (-\pi,0)\cup(0,\pi)\right)  \backslash A_{n},$
the system $\{\Psi_{n,t}:n\in\mathbb{Z}\}$ is complete we have $\alpha
_{n}(t)\neq0$. Hence $\left\vert \frac{1}{\alpha_{n}}\right\vert $ is
continuous at $\left(  (-\pi,0)\cup(0,\pi)\right)  \backslash A_{n}$. The last
statement of the lemma follows from the fact that the sets $B,$ $C,$ $B_{1},$
$C_{1}$and $A_{n}$ are finite.
\end{proof}

Using Lemma 1 we prove the following

\begin{proposition}
Let $\mathbb{E},$ $\mathbb{S},$ and $\mathbb{M}$ be respectively the sets of
ESS, spectral singularities and multiple eigenvalues of $L_{t}(q)$ for
$t\in(-\pi,\pi].$ Then $\mathbb{E}\subset\mathbb{S}\subset\mathbb{M}.$
\end{proposition}

\begin{proof}
If $\lambda\in\sigma(L_{t_{0}}(q))$ is not a spectral singularity then by
Proposition 1(a), $\frac{1}{\alpha_{k}}$ is bounded in some deleted
neighborhood $D(t_{0},\varepsilon)$ of $t_{0}$ for all indices $k$ such that
$\lambda_{k}(t_{0})=\lambda,$ where
\begin{equation}
D(t_{0},\varepsilon)=[t_{0}-\varepsilon,t_{0})\cup(t_{0},t_{0}+\varepsilon],
\end{equation}
and by Lemma 1, it is piecewise continuous. Therefore $\frac{1}{\alpha_{k}}$
is integrable in $D(t_{0},\varepsilon),$ and hence, by Definition 3,
$\lambda\notin\mathbb{E}$. The inclusion $\mathbb{S}\subset\mathbb{M}$ is
well-known (see [4, 15]).
\end{proof}

Proposition 2 shows that to study ESS we need to investigate the integral
\begin{equation}%
{\textstyle\int\limits_{D(t_{0},\varepsilon)}}
\frac{1}{\alpha_{k}(t)}dt
\end{equation}
when $\lambda_{k}(t_{0})$ is a multiple eigenvalue. Let $\Lambda$ be a
multiple eigenvalues of the operator $L_{t_{0}}.$ The set
\begin{equation}
\mathbb{T}(\Lambda)=:\left\{  k\in\mathbb{Z}:\lambda_{k}(t_{0})=\Lambda
\right\}
\end{equation}
is finite since the multiplicity of the eigenvalues of $L_{t_{0}}$ is finite.
Let $p(\Lambda)$ be the multiplicity of the eigenvalue $\Lambda$ of $L_{t_{0}%
}.$ It follows from (4) that $\Lambda$ is also a multiple eigenvalues of
multiplicity $p(\Lambda)$ of $L_{-t_{0}}$ too and $\Lambda\notin\sigma(L_{t})$
for $t\neq\pm t_{0}.$

Now, taking into account the continuity of $\lambda_{k}$, using the implicit
function theorem for (4) and then the definitions 1 and 3 we prove the following

\begin{proposition}
Let $\lambda_{0}$ be a multiple eigenvalue of $L_{t_{0}}$ of multiplicity
$m>1.$

$(a)$ There exists $\varepsilon>0$ such that for $k\in\mathbb{T}(\lambda_{0})$
and $t\in D(t_{0},\varepsilon)$ the eigenvalues $\lambda_{k}(t)$ are simple
and the followings hold
\begin{align}
\lambda_{k}(t)-\lambda_{0}  &  \sim(t-t_{0})^{\frac{1}{m}}\text{ if }t_{0}%
\neq0,\pi,\\
\lambda_{k}(t)-\lambda_{0}  &  \sim(t-t_{0})^{\frac{2}{m}}\text{ if }%
t_{0}=0,\pi
\end{align}
as $t\rightarrow t_{0},$ where $f(t)\sim g(t)$ as $t\rightarrow t_{0}$ means
that there exist positive constants $\delta,$ $c_{1}$ and $c_{2}$ such that
$c_{1}\left\vert g(t)\right\vert \leq\left\vert f(t)\right\vert \leq
c_{2}\left\vert g(t)\right\vert $ for all $t\in D(t_{0},\delta).$

$(b)$ Let $\alpha_{k}(t)\sim(t-t_{0})^{\beta}$ for all $k\in\mathbb{T}%
(\lambda_{0}),$ where $\beta$ is a real number. If $\beta=0$ then $\lambda
_{0}$ is not a spectral singularity. If $0<\beta<1$ then $\lambda_{0}$ is a
spectral singularity of $L$ and is not an ESS. If $\beta\geq1$ then
$\lambda_{0}$ is an ESS.
\end{proposition}

\begin{proof}
Using the Taylor formula for $F(\lambda)$ and $\cos t$ and taking into account that

$F(\lambda_{0})=2\cos t_{0},$ $F^{(k)}(\lambda_{0})=0$ for $k=1,2,....(m-1)$
and $F^{^{(m)}}(\lambda_{0})\neq0$ we obtain
\[
F(\lambda)=2\cos t_{0}+F^{^{(m)}}(\lambda_{0})(\lambda-\lambda_{0}%
)^{m}(1+o(1))
\]
as $\lambda\rightarrow\lambda_{0}$ and
\[
2\cos t=2\cos t_{0}-\left(  \sin t_{0}\right)  (t-t_{0})-(t-t_{0})^{2}\left(
\tfrac{1}{2}+o(1)\right)
\]
as $t\rightarrow t_{0}.$ These equalities with the equality $F(\lambda
_{k}(t))=2\cos t$ and the continuity of $\lambda_{k}$ give the proof of $(a).$
The proof of $(b)$ follows from the definitions 1 and 3.
\end{proof}

Now we are ready to prove the main results. First, let us consider the case
$t_{0}\neq0,\pi.$

\begin{theorem}
If $t_{0}\in(0,\pi)$ and $\lambda_{0}$ is a multiple eigenvalue of $L_{t_{0}%
},$ then $\lambda_{0}$ is a spectral singularity of $L$ and is not an ESS.
\end{theorem}

\begin{proof}
Let $\lambda_{0}$ be a multiple eigenvalue of multiplicity $m>1.$ Then
$F^{^{\prime}}(\lambda)\sim(\lambda_{0}-\lambda)^{m-1}$ as $\lambda
\rightarrow\lambda_{0}$. Hence, using Proposition 3(a) and taking into account
that $\lambda_{k}$ for $k\in\mathbb{T}(\lambda_{0})$ is continuous at $t_{0},$
that is, for each neighborhood $U$ of $\lambda_{0}$ there exist a neighborhood
$\delta\subset(-\pi,0)\cup(0,\pi)$ of $t_{0}$ such that $\lambda_{k}%
(\delta)\subset U$ we obtain
\begin{equation}
F^{^{\prime}}(\lambda_{k}(t))\sim(t_{0}-t)^{\frac{m-1}{m}},\text{ }\forall
k\in\mathbb{T}(\lambda_{0})
\end{equation}
as $t\rightarrow t_{0}.$ Now to prove the theorem we use Proposition 3(b) and
the formula
\begin{equation}
\alpha_{k}(t)=\frac{-\varphi F^{^{\prime}}(\lambda)}{\left\Vert \varphi
\theta(\cdot,\lambda)+\tfrac{1}{2}(\varphi^{\prime}-\theta-ip(\lambda
))\varphi(\cdot,\lambda)\right\Vert \left\Vert \varphi\theta(\cdot
,\lambda)+\tfrac{1}{2}(\varphi^{\prime}-\theta+ip(\lambda))\varphi
(\cdot,\lambda)\right\Vert }%
\end{equation}
for $\lambda=\lambda_{k}(t)$ obtained from (29), (36) and (34). Consider two cases:

Case 1: $\varphi(\lambda_{0})\neq0.$ Then $\varphi(\lambda)\sim1$ as
$\lambda\rightarrow\lambda_{0},$ since $\varphi$ is an entire functions. On
the other hand $\varphi^{\prime}-\theta\pm ip=O(1)$ as $\lambda\rightarrow
\lambda_{0}.$ Using this and taking into account that $\theta(\cdot,\lambda)$
and $\varphi(\cdot,\lambda)$ are linearly independent elements of $L_{2}(0,1)$
(see the proof of Lemma 1(c)) we obtain%
\begin{equation}
\left\Vert \varphi\theta(x,\lambda)+\tfrac{1}{2}(\varphi^{\prime}-\theta\pm
ip(\lambda))\varphi(x,\lambda)\right\Vert \sim1
\end{equation}
as $\lambda\rightarrow\lambda_{0}.$ Therefore using (62) in (63) we obtain
\begin{equation}
\alpha_{k}(t)\sim(t_{0}-t)^{\frac{m-1}{m}}%
\end{equation}
as $t\rightarrow t_{0}$ for all $k\in\mathbb{T}(\lambda_{0}).$ Thus, \ by
Proposition 3(b), $\lambda_{0}$ is a spectral singularity of $L$ and is not an ESS.

Case 2: $\varphi(\lambda_{0})=0.$ Then there exists a positive integer $s$
such that
\begin{equation}
\varphi(\lambda)\sim(\lambda-\lambda_{0})^{s}\text{ }%
\end{equation}
as $\lambda\rightarrow\lambda_{0}.$ On the other hand, by (35) and (30) we
have
\begin{equation}
\left(  \varphi^{\prime}(\lambda)-\theta(\lambda)+ip(\lambda)\right)  \left(
\varphi^{\prime}(\lambda)-\theta(\lambda)-ip(\lambda)\right)  =\left(
\varphi^{\prime}(\lambda)-\theta(\lambda)\right)  ^{2}+
\end{equation}%
\[
4-\left(  \varphi^{\prime}(\lambda)+\theta(\lambda)\right)  ^{2}%
=4-4\varphi^{\prime}(\lambda)\theta(\lambda)=-4\varphi(\lambda)\theta
^{^{\prime}}(\lambda).
\]
Since $p(\lambda_{0})=\sin t_{0}\neq0,$ at least one the numbers
$\varphi^{\prime}(\lambda_{0})-\theta(\lambda_{0})+ip(\lambda_{0})$ \ and

$\varphi^{\prime}(\lambda_{0})-\theta(\lambda_{0})-ip(\lambda_{0})$ is not
zero. Suppose, without loss of generality, the first of them is not zero. Then
using (67) and (66) and arguing as in the proof of (64) we get
\begin{align*}
\varphi^{\prime}(\lambda)-\theta(\lambda)-ip(\lambda)  &  =O((\lambda
-\lambda_{0})^{s}),\text{ }\\
\left\Vert \varphi\theta(x,\lambda)+\tfrac{1}{2}(\varphi^{\prime}%
-\theta-ip(\lambda))\varphi(x,\lambda)\right\Vert  &  \sim(\lambda-\lambda
_{0})^{s},\\
\left\Vert \varphi\theta(x,\lambda)+\tfrac{1}{2}(\varphi^{\prime}%
-\theta+ip(\lambda))\varphi(x,\lambda)\right\Vert  &  \sim1
\end{align*}
as $\lambda\rightarrow$ $\lambda_{0}.$ Now, from (63), (62) and (66) we obtain
(65). Therefore the proof of the theorem follows from Proposition 3(b)
\end{proof}

By the similar arguments one can find conditions on $\lambda_{n}(t)$ for
$t=0,\pi$ to be or not to be the ESS. Here we prove only one criterion for
large value of $n$ which will be used essentially for the spectral expansion.
For this we use the following well-known statements (see [6]). The large
eigenvalues of the Dirichlet and Neimann boundary value problems are simple,
that is, the multiplicities of the large roots $\ \lambda_{0}$ and $\mu_{0}$
of $\varphi(\lambda)=0$ and $\theta(\lambda)=0$ is $1.$ It mean that%
\begin{equation}
\varphi(\lambda)\sim(\lambda-\lambda_{0})\text{ }\And\text{\ }\theta
(\lambda)\sim(\lambda-\mu_{0})
\end{equation}
as $\lambda\rightarrow\lambda_{0}$ and $\lambda\rightarrow\mu_{0}$ respectively.

Similarly, if $\left\vert k\right\vert \gg1$ and $\lambda_{0}=\lambda_{k}(0)$
is the multiple eigenvalue of $L_{0}$ then it is double eigenvalue and hence
$F(\lambda_{0})=2,$ $F^{\prime}(\lambda_{0})=0,$ $F^{\prime\prime}(\lambda
_{0})\neq0$ which implies that
\begin{equation}
F^{\prime}(\lambda)\sim(\lambda-\lambda_{0})\text{ }\And\text{\ }%
p(\lambda)\sim(\lambda-\lambda_{0})
\end{equation}
as $\lambda\rightarrow\lambda_{0}.$ Therefore by (61) we have
\begin{equation}
\lambda_{k}(t)-\lambda_{0}\sim t,\text{ }F^{\prime}(\lambda_{k}(t))\sim
t,\text{ }\forall k\in\mathbb{T}(\lambda_{0})
\end{equation}
as $t\rightarrow0.$ Now using (70) we prove the following main result of this section

\begin{theorem}
Let $\lambda_{0}$ be a large and multiple eigenvalue of $L_{0}.$ Then the
following statements are equivalent:

$(a)$ The eigenvalue $\lambda_{0}$ of $L_{0}$ is an ESS of $L.$

$(b)$ The eigenvalue $\lambda_{0}$ is a spectral singularity of $L.$

$(c)$ The geometric multiplicity of the eigenvalue $\lambda_{0}$ is $1,$ that
is, there exist one eigenfunction and one associated function corresponding to
$\lambda_{0}$.

$(d)$ $\lambda_{0}$ is neither Dirichlet nor Naimann eigenvalue, that is,
\begin{equation}
\varphi(\lambda_{0})\neq0\text{ }\And\text{\ }\theta^{^{\prime}}(\lambda
_{0})\neq0.
\end{equation}

The theorem continues to hold if $L_{0}$ is replaced by $L_{\pi}.$
\end{theorem}

\begin{proof}
We prove the theorem for $L_{0}.$ The proof of the case $L_{\pi}$ is the same.
First let us prove that $(a)$ and $(b)$ hold if and only if $\varphi
(\lambda_{0})\neq0.$ If the last inequality holds then (64) holds too.
Therefore using (64) and (70) in (63) we obtain that
\begin{equation}
\alpha_{k}(t)\sim t
\end{equation}
as $t\rightarrow0$ for all $k\in\mathbb{T}(\lambda_{0})$ and hence, by
propositions 3(b) and 2, $(a)$ and $(b)$ hold.

Now suppose that $\varphi(\lambda_{0})=0$. Then using (30) and the equality
$F(\lambda_{0})=2$ by direct calculation we obtain $\theta(\lambda
_{0})=1=\varphi^{^{\prime}}(\lambda_{0})=1.$ Therefore, the first relation of
(68) and the second relation of (69) imply that
\begin{equation}
\left\Vert \varphi\theta(x,\lambda)+\tfrac{1}{2}(\varphi^{\prime}-\theta\pm
ip(\lambda))\varphi(x,\lambda)\right\Vert \sim(\lambda-\lambda_{0})
\end{equation}
and by (63), (68)-(70) we have $\alpha_{k}(t)\sim1$ as $t\rightarrow0$ for all
$k\in\mathbb{T}(\lambda_{0}),$ that is, $(a)$ and $(b)$ does not hold.

Instead of (29) using (32) in the same way we prove that $(a)$ and $(b)$ hold
if and only if $\theta^{^{\prime}}(\lambda_{0})\neq0.$ Thus we proved that
$(a)$, $(b)$ and $(d)$ are equivalent.

To complete the proof of the theorem we prove that $(d)\Longrightarrow(c)$ and
$(c)\Longrightarrow(b).$ Suppose that $(d)$ and hence (71) holds. If $(c)$
does not hold then both solution $\varphi(x,\lambda_{0})$ and $\theta
(x,\lambda_{0})$ are periodic function. In this case by (5) we have
$\varphi(\lambda)=0$ which contradicts (71). Thus $(d)\Longrightarrow(c)$.

If $(c)$ holds, then, there is one eigenfnction $\Psi_{n,0}$ corresponding to
the eigenvalue $\lambda_{n}(0)=\lambda_{0}$ and an associated function $\phi,$
satisfying
\begin{equation}
(L_{0}-\lambda_{n}(0))\phi=\Psi_{n,0}.
\end{equation}
Multiplying both sides of (74) by $\Psi_{n,0}^{\ast}$ we obtain $\alpha
_{n}(0)=0.$ On the other hand, if the geometric multiplicity of $\lambda
_{n}(0)$ is $1$ then by Lemma 1(b) $\alpha_{n}$ is continuous at $0.$ Hence
$\alpha_{n}(t)\rightarrow0$ as $t\rightarrow0$ and by Proposition 1(a)
$\lambda_{0}$ is a spectral singularity, that is, $(b)$ holds.
\end{proof}

If the geometric multiplicity of $\lambda_{n}(0)$ is $1,$ then at least one of
$\varphi(\lambda_{n}(0))$ and $\theta^{^{\prime}}(\lambda_{n}(0))$ is not
zero. Indeed if both are zero, then it follows from (30), (4) and (5) that
both $\varphi(x,\lambda_{n}(0))$ and $\theta(x,\lambda_{n}(0))$ are
eigenfunctions which contradicts the assumption. Without loss of generality,
assume that $\varphi(\lambda_{n}(0))$ is not zero. Then (64) holds. Therefore
from (63) and (61) we obtain that
\begin{equation}
\alpha_{n}(t)\sim t^{\frac{2(m-1)}{m}}%
\end{equation}
which implies the following

\begin{proposition}
If $\lambda_{n}(0),$ where $n\in\mathbb{Z},$ is a multiple eigenvalue with
geometric multiplicity $1,$ then it is an ESS. The statement continuous to
hold if $\lambda_{n}(0)$ is replaced by $\lambda_{n}(\pi).$
\end{proposition}

Now we use the following classical result (see p.8-9 of [5] and p.34-35 of [1]):

\textit{ If }$q$\textit{ is an even function, then the eigenvalues of }$L_{0}%
$\textit{ and }$L_{\pi}$\textit{ are either Dirichlet or Naimann eigenvalues.}

Note that in [1, 5] this result were proved for the real-valued potentials.
However, the proof pass through for the complex-valued potentials without any
change. \ This classical result shows that Theorem 2(d) does not hold for the
even potentials. Therefore, Theorem 2, Proposition 2 and Theorem 1 immediately
imply the following:

\begin{corollary}
If the potential $q$ is an even function, then

$(a)$ The operators $L_{0}(q)$\textit{ and }$L_{\pi}(q)$ have no associated
functions corresponding to the large eigenvalues.

$(b)$The large eigenvalues of $L_{0}(q)$\textit{ and }$L_{\pi}(q)$ are not
spectral singularities.

$(c)$ The operator $L(q)$ may have only finite number of ESS.\textit{ }
\end{corollary}

\section{Spectral expansion}

In this section we construct spectral expansion by using (47). The term by
term integration in (47) was proved in the papers [14, 16, 18] for the curve
$l$ satisfying (42). Here we prove it for a little different curve and by the
other method for the independence of this paper. For this first let us
construct the suitable curve of integration by taking into account the results
of Section 2. Since, only the eigenvalues $\lambda_{n}(0)$ and $\lambda
_{n}(\pi)$ may became the ESS, we choose the curve of integration so that it
only pass over the points $0$ and $\pi$. Namely, we construct the curve of
integration as follows. Let $h$ be positive number such that
\begin{equation}
F^{^{\prime}}(\lambda_{n}(t))\neq0,\text{ }\varphi(\lambda_{n}(t))\neq0,\text{
}\forall t\in\left(  \gamma(0,h)\cup\gamma(\pi,h)\right)  ,\text{ }\forall
n\in\mathbb{Z}\text{,}%
\end{equation}
where $\gamma(0,h)$ and $\gamma(\pi,h)$ are the semicircles
\begin{equation}
\gamma(0,h)=\{\left\vert t\right\vert =h,\operatorname{Im}t\geq0\},\text{
}\gamma(\pi,h)=\{\left\vert t-\pi\right\vert =h,\operatorname{Im}t\geq0\}.
\end{equation}
Since the accumulation points of the roots of the equations $F^{^{\prime}%
}(\lambda_{n}(t))=0,$ $\varphi(\lambda_{n}(t))=0$ are $0$ and $\pi$ there
exist $\gamma(0,h)$ and $\gamma(\pi,h)$ satisfying (76). Define $l(h)$ by
\begin{equation}
l(h)=B(h)\cup\gamma(0,h)\cup\gamma(\pi,h)),
\end{equation}
where $B(h)=[h,\pi-h]\cup\lbrack\pi+h,2\pi-h].$ Thus $l(h)$ consist of the
intervals $[h,\pi-h]$ and $[\pi+h,2\pi-h]$ and semicircles (77). Denote the
points of $A\cap B(h)$ by $t_{1},t_{2},...,t_{s}$ and put
\begin{equation}
E(h)=B(h)\backslash\left\{  t_{1},t_{2},...,t_{s}\right\}  ,
\end{equation}
where $\ A$ is defined in (15) and $A\cap B(h)$ is a finite set because the
accumulation points of $A$ are $0$ and $\pi.$ In (47) instead of $l$ using
$l(h)=B(h)\cup\gamma(0,h)\cup\gamma(\pi,h)$ and taking into account that
integral over $B(h)$ is equal to the integral over $E(h)$ we obtain%
\begin{equation}
f=\frac{1}{2\pi}\left(  \int\limits_{E(h)}f_{t}(x)dt+\int\limits_{\gamma
(0,h)}f_{t}(x)dt+\int\limits_{\gamma(\pi,h)}f_{t}(x)dt\right)
\end{equation}
and by (44)
\begin{align}
\int\limits_{E(h)}f_{t}(x)dt  &  =\int\limits_{E(h)}\sum\limits_{k\in
\mathbb{Z}}a_{k}(t)\Psi_{k,t}(x)dt,\\
\int\limits_{\gamma(0,h)}f_{t}(x)dt  &  =\int\limits_{\gamma(0,h)}%
\sum\limits_{k\in\mathbb{Z}}a_{k}(t)\Psi_{k,t}(x)dt,\int\limits_{\gamma
(\pi,h)}f_{t}(x)dt=\int\limits_{\gamma(\pi,h)}\sum\limits_{k\in\mathbb{Z}%
}a_{k}(t)\Psi_{k,t}(x)dt.\nonumber
\end{align}
Now we prove that the series in (81) can be integrated term by term. First we
prove it for the first integral (see Theorem 3) and then for the second and
third integral (see Theorem 4), that is, first we prove the following
\begin{equation}
\int\limits_{E(h)}\sum\limits_{k\in\mathbb{Z}}a_{k}(t)\Psi_{k,t}dt=\sum
_{n\in\mathbb{Z}}\int\limits_{E(h)}a_{n}(t)\Psi_{n,t}(x)dt.
\end{equation}
For this we show that the integrals in the right-hand side of (82) exists (see
Proposition 4) and then estimate the remainders
\begin{equation}
R_{n}(x,t)=\sum_{k>n}a_{k}(t)\Psi_{k,t}(x),\text{ }R_{-n}(x,t)=\sum
_{k<-n}a_{k}(t)\Psi_{k,t}(x)
\end{equation}
(see Lemma 2) of the series
\begin{equation}
\sum_{k\in\mathbb{Z}}a_{k}(t)\Psi_{k,t}(x).
\end{equation}

\begin{proposition}
Let $f$ be continuous and compactly supported function and $\delta\in(0,1).$

$(a)$ For each $n\in\mathbb{Z}$ the integral
\begin{equation}
\int\limits_{E(\delta)}a_{n}(t)\Psi_{n,t}(x)dt
\end{equation}
exists, where $E(\delta)$ and $a_{n}(t)$ are defined in (79) and (44).

$(b)$ If $\lambda_{n}(0)$ and $\lambda_{n}(\pi)$ are not ESS then the
integrals
\begin{equation}
\int\limits_{(-\delta,\delta)}a_{n}(t)\Psi_{n,t}(x)dt\text{ }\And\text{ }%
\int\limits_{(\pi-\delta,\pi+\delta)}a_{n}(t)\Psi_{n,t}(x)dt
\end{equation}
exist respectively.
\end{proposition}

\begin{proof}
$(a)$ Theorem 1 with the Definition 3 implies that $\tfrac{1}{\alpha_{n}}$ is
integrable on $E(h).$ Using the definitions of $a_{n}(t)$ and $f_{t}$ (see
(40)) and Schwarz inequality and taking into account that $f$ is a continuous
and compactly supported function we obtain that there exists a number $M$ such
that
\begin{equation}
\left\vert a_{n}(t)\right\vert \leq M\left\vert \tfrac{1}{\alpha_{n}%
(t)}\right\vert ,\forall t\in(-\pi,0)\cup(0,\pi).
\end{equation}
On the other hand, it follows from Lemma 1 that $a_{n}(t)$ is a piecewise
continuous function and for each fixed $x$, $\Psi_{n,t}(x)$ is a piecewise
continuous and bounded function on $E(h).$ Therefore the integral (85) exists.

$(b)$ If $\lambda_{n}(0)$ is not ESS then by Definition 3, $\tfrac{1}%
{\alpha_{n}}$ is integrable on $(-\varepsilon,\varepsilon)$ for some
$\varepsilon>0.$ Therefore using (87) and arguing as in the proof of $(a)$ we
see that the first integral in (86) exists. In the same way we prove that the
second integral exists too.
\end{proof}

\begin{remark}
Let $E$ be a subset of $L_{2}(-\infty,\infty)$ such that if $f\in E,$ then the
norm $\left\Vert f_{t}\right\Vert $ of the Gelfand transform $f_{t}%
(x)=\Upsilon f$ $(t)$, defined by (40), is bounded in $(-\pi,\pi]$ almost
everywhere. Then, by the Schwarz inequality, (87) holds almost everywhere.
Therefore the proof of the Proposition 5 shows that if $\tfrac{1}{\alpha
_{n}(t)}$ is integrable over the measurable subset $I$ of $(-\pi,\pi]$ then
$a_{n}(t)\Psi_{n,t}(x)$ is also integrable in $I$ for each $x\in\lbrack0,1].$

Now, conversely, suppose that $\tfrac{1}{\alpha_{n}(t)}$ is not integrable
over $I.$ By the definition of $E,$ the equality $\left\Vert \Psi_{n,t}^{\ast
}\right\Vert =1$ implies that $\Upsilon^{-1}\Psi_{n,t}^{\ast}\in E,$ where
$\Upsilon^{-1}$ is the inverse \ Gelfand transform. Let $f$ $=\Upsilon
^{-1}\Psi_{n,t}^{\ast}$ . Then $a_{n}(t)=\tfrac{1}{\alpha_{n}(t)}$. Therefore
using Lemma 1 one can easily show that $a_{n}(t)\Psi_{n,t}(x)$ is not
integrable on $I$ for some $x\in\lbrack0,1].$
\end{remark}

Now we estimate (83), by using the following uniform with respect to $t$ in
$E(h)$ asymptotic formulas
\begin{equation}
\Psi_{n,t}(x)=e^{i(2\pi n+t)x}+h_{n,t}(x),\text{ }\left\Vert h_{n,t}%
\right\Vert =O(n^{-1}),
\end{equation}%
\begin{equation}
\Psi_{n,t}^{\ast}(x)=e^{i(2\pi n+t)x}+h_{n,t}^{\ast}(x),\text{ }\left\Vert
h_{n,t}^{\ast}\right\Vert =O(n^{-1}),\text{ }\tfrac{1}{\alpha_{n}%
(t)}=1+O(n^{-1})
\end{equation}
(see (17), (19) and (24)).

\begin{lemma}
There exist a positive constants $N$ and $c,$ independent of $t,$ such that
\begin{equation}
\parallel R_{n}(\cdot,t)\parallel^{2}\leq c\left(
{\textstyle\sum\limits_{k>n}}
\mid(f_{t},e^{i(2\pi k+t)x})\mid^{2}+\frac{1}{n}\right)
\end{equation}
for $n>N$ and $t\in E(h).$
\end{lemma}

\begin{proof}
During the proof of the lemma we denote by $c_{1},c_{2},...$ the positive
constants that do not depend on $t.$ They will be used in the sense that there
exists $c_{i}$ such that the inequality holds. To prove (90) first we prove
the inequality
\begin{equation}
\sum\limits_{k>n}\mid a_{k}(t)\mid^{2}\leq c_{1}\left(  \sum\limits_{k>n}%
\mid(f_{t},e^{i(2\pi k+t)x})\mid^{2}+\frac{1}{n}\right)  ,
\end{equation}
where $a_{k}(t)$\ is defined in (44), and then the equality
\begin{equation}
\parallel R_{n}(.,t)\parallel^{2}=\left(  1+O(n^{-1})\right)  \sum
\limits_{k>n}\mid a_{k}(t)\mid^{2}.
\end{equation}
It follows from (89) that
\begin{equation}
\mid a_{k}(t)\mid^{2}\leq8\mid(f_{t},e^{i(2\pi k+t)x})\mid^{2}+8\mid
(f_{t},h_{n,t}^{\ast})\mid^{2}.
\end{equation}
Since $f$ is a compactly supported and continuous function we have
\begin{equation}
\parallel f_{t}\parallel^{2}<c_{2}.
\end{equation}
It with the Schwarz inequality and second equality of (89) implies that%
\begin{equation}
\mid(f_{t},h_{n,t}^{\ast})\mid^{2}<c_{3}n^{-2}.
\end{equation}
Therefore (91) follows from (93).

Now we prove (92). Since $\left\{  e^{i2\pi kx}:k\in\mathbb{Z}\right\}  $ is
an orthonormal basis, using the Bessel inequality and (94) we obtain
\[
\sum_{k:\mid k\mid>N}\mid(f_{t},e^{i(2\pi k+t)x})\mid^{2}\leq\parallel
f_{t}\parallel^{2}<c_{2}.
\]
Hence, it follows from (91) that
\[
\sum_{k:\mid k\mid>n}\mid a_{k}(t)\mid^{2}\leq c_{4}%
\]
and by (88), $\left\Vert a_{k}(t)h_{k,t}(x)\right\Vert \leq\mid a_{k}%
(t)\mid^{2}+c_{5}n^{-2}.$ Therefore the series
\[
\sum_{k>n}a_{k}(t)e^{i(2\pi k+t)x}\text{ }\And\text{\ }\sum_{k>n}%
a_{k}(t)h_{k,t}(x)
\]
converge in the norm of $L_{2}(0,1)$ and we have
\begin{equation}
\parallel R_{n}(.,t)\parallel^{2}=\left\Vert \sum_{k>n}a_{k}(t)e^{i(2\pi
k+t)x}+\sum_{k>n}a_{k}(t)h_{k,t})\right\Vert ^{2}\leq2S_{1}+2S_{2}^{2}%
\end{equation}
where
\begin{equation}
S_{2}=\left\Vert \sum_{k>n}a_{k}(t)h_{k,t}\right\Vert
\end{equation}
and
\begin{equation}
S_{1}=\left\Vert \sum_{k>n}a_{k}(t)e^{i(2\pi k+t)x}\right\Vert ^{2}%
=\sum\limits_{k>n}\mid a_{k}(t)\mid^{2}.
\end{equation}
\ Now let us estimate $S_{2}.$ It follows from the second equality of (88)
that
\[
S_{2}\leq c_{6}\sum_{k>n}\mid a_{k}(t)\mid\frac{1}{\mid n\mid}.
\]
Therefore using the Schwarz inequality for $l_{2}$ we obtain
\begin{equation}
S_{2}^{2}=\left(  \sum_{k>n}\mid a_{k}(t)\mid^{2}\right)  O(n^{-1}).
\end{equation}
Thus (92) follows from (96), (98) and \ (99). It with (91) yields the proof of
the lemma
\end{proof}

Now we are ready to prove the following

\begin{theorem}
For every compactly supported and continuous function $f$ the equality
\begin{equation}
\int\limits_{E(h)}f_{t}(x)dt=%
{\displaystyle\sum\limits_{k\in\mathbb{Z}}}
\int\limits_{E(h)}a_{k}(t)\Psi_{k,t}(x)dt
\end{equation}
holds, where $0<h<\frac{1}{15\pi}.$ The series in (100) converges in the norm
of $L_{2}(a,b)$ for every $a,b\in\mathbb{R}.$
\end{theorem}

\begin{proof}
By (41) and (45) we have $R_{n}(x+1,t)=e^{it}R_{n}(x,t).$ Therefore it follows
from (90) that
\begin{equation}
\parallel R_{n}(\cdot,t)\parallel_{(-m,m)}^{2}\leq2mc\left(
{\textstyle\sum\limits_{k>n}}
\mid(f_{t},e^{i(2\pi k+t)x})\mid^{2}+\frac{1}{n}\right)  ,
\end{equation}
where $\parallel f\parallel_{(-m,m)}$ is the $L_{2}(-m,m)$ norm of $f$. Since
the sequence
\[
\left\{
{\textstyle\sum\limits_{k>n}}
\mid(f_{t},e^{i(2\pi k+t)x})\mid^{2}:\text{ }n=1,2,...,\right\}
\]
of the continuous nonincreasing functions converges uniformly to zero on
$[-\pi,\pi]$ it follows from (101) that $\parallel R_{n}(\cdot,t)\parallel
_{(-m,m)}^{2}$ also converges to zero uniformly on $E(h)$ as $n\rightarrow
\infty.$ It implies that
\begin{equation}
\int\limits_{E(h)}\int\limits_{(-m,m)}\mid R_{n}(x,t)\mid^{2}dxdt\rightarrow0
\end{equation}
as $n\rightarrow\infty.$ Now using the obvious inequality $\left\vert \int
_{E}f(t)dt\right\vert ^{2}\leq2\pi\int_{E}\left\vert f(t)\right\vert ^{2}dt$
and (83), (102), we obtain%
\begin{equation}
\left\Vert \int\limits_{E}\sum_{k>n}a_{k}(t)\Psi_{k,t}dt\right\Vert
_{(-m,m)}^{2}\leq2\pi\int\limits_{(-m,m)}\int\limits_{E}\left\vert \sum
_{k>n}a_{k}(t)\Psi_{k,t}(x)\right\vert ^{2}dtdx\rightarrow0
\end{equation}
as $n\rightarrow\infty$. Thus we have
\begin{equation}
\int\limits_{E}\sum_{k>N}a_{k}(t)\Psi_{k,t}(x)dt=\sum_{k>N}\int\limits_{E}%
a_{k}(t)\Psi_{k,t}(x)dt,
\end{equation}
where the last series converges in the norm of $L_{2}(-m,m)$ for every
$m\in\mathbb{N}.$ In the same way we prove that
\begin{equation}
\int\limits_{E}\sum_{k<-N}a_{k}(t)\Psi_{k,t}(x)dt=\sum_{k<-N}\int
\limits_{E}a_{k}(t)\Psi_{k,t}(x)dt.
\end{equation}
\ Therefore using (104), (105) and Proposition 5(a) we get the proof of the theorem
\end{proof}

Now let us consider the term by term integration of the second and third
integral in (81). Using the conditions in (76) and arguing as in the proof of
Lemma 1 we see that for each $x\in\lbrack0,1],$ $\Psi_{n,t}(x)$ and
$\Psi_{n,t}^{\ast}(x)$ are continuous and bounded on $\gamma(0,h)\cup
\gamma(\pi,h).$ Therefore, Proposition 5 (a) continues to hold if we replace
$E(\delta)$ by $\gamma(0,h)$ and $\gamma(\pi,h).$ Similarly instead of (88)
and (89) and the orthonormal basis $\left\{  e^{i(2\pi k+t)x}:k\in
\mathbb{Z}\right\}  $ using (17) and (19) and $\left\{  \frac{1}{\parallel
e^{itx}\parallel}e^{i(2n\pi+t)x}:k\in\mathbb{Z}\right\}  $ we see that Lemma 2
continues to hold if we replace $E(h)$ by $\gamma(0,h)$ and $\gamma(\pi,h).$
In the same way we prove (102) when $E(h)$ is replaced by $\gamma(0,h)$ and
$\gamma(\pi,h).$ Thus repeating the proof of Theorem 3 we obtain

\begin{theorem}
For every compactly supported and continuous function $f$ the equalities
\begin{equation}
\int\limits_{\gamma(0,h)}f_{t}(x)dt=\sum_{k\in\mathbb{Z}}\int\limits_{\gamma
(0,h)}a_{k}(t)\Psi_{k,t}(x)dt
\end{equation}
and
\begin{equation}
\int\limits_{\gamma(\pi,h)}f_{t}(x)dt=\sum_{k\in\mathbb{Z}}\int\limits_{\gamma
(\pi,h)}a_{k}(t)\Psi_{k,t}(x)dt
\end{equation}
hold, where $0<h<\frac{1}{15\pi}$ and (76) holds. The series in (106) and
(107) converges in the norm of $L_{2}(a,b)$ for every $a,b\in\mathbb{R}.$
\end{theorem}

Now to prove the expansion theorem we try to replace $\gamma(0,h)$ and
$\gamma(\pi,h)$ by $[-h,h]$ and $[\pi-h,\pi+h]$ in the right hand sides of
(106) and (107) respectively.

\begin{theorem}
Let $f$ be continuous and compactly supported function, $0<h<\frac{1}{15\pi}$
and (76) holds. Then the following equalities hold
\begin{equation}
\int\limits_{\gamma(0,h)}f_{t}(x)dt=\int\limits_{[-h,h]}\left(  \sum
_{\left\vert n\right\vert \leq N(h)}a_{n}(t)\Psi_{n,t}(x)\right)  dt+
\end{equation}%
\[
\sum_{n>N(h)}\int\limits_{[-h,h]}\left(  a_{n}(t)\Psi_{n,t}(x)+a_{-n}%
(t)\Psi_{-n,t}(x)\right)  dt,
\]%
\begin{equation}
\int\limits_{\lbrack-h,h]}\left(  \sum_{\left\vert n\right\vert \leq N_{h}%
(0)}a_{n}(t)\Psi_{n,t}(x)\right)  dt=\lim_{\delta\rightarrow0}\left(
\sum_{\left\vert n\right\vert \leq N_{h}(0)}\int\limits_{\delta<\left\vert
t\right\vert \leq h}a_{n}(t)\Psi_{n,t}(x)dt\right)  ,
\end{equation}%
\begin{equation}
\int\limits_{\lbrack-h,h]}\left(  a_{n}(t)\Psi_{n,t}+a_{-n}(t)\Psi
_{-n,t}\right)  dt=\lim_{\delta\rightarrow0}\left(  \int\limits_{\delta
<\left\vert t\right\vert \leq h}a_{n}(t)\Psi_{n,t}dt+\int\limits_{\delta
<\left\vert t\right\vert \leq h}a_{-n}(t)\Psi_{-n,t}dt\right)  ,
\end{equation}
where $N(h)$ is defined in introduction (see \textbf{(a)}). Moreover, if
$\lambda_{n}(0),$ where $n>N(h),$ is not an ESS then
\begin{equation}
\int\limits_{\lbrack-h,h]}\left(  a_{n}(t)\Psi_{n,t}+a_{-n}(t)\Psi
_{-n,t}\right)  dt=\int\limits_{[-h,h]}a_{n}(t)\Psi_{n,t}dt+\int
\limits_{[-h,h]}a_{-n}(t)\Psi_{-n,t}dt.
\end{equation}
The series in (108) converges in the norm of $L_{2}(a,b)$ for every
$a,b\in\mathbb{R}.$
\end{theorem}

\begin{proof}
It follows from (11) that for $n>N(h)$ the circle

$C(n)=\left\{  z\in\mathbb{C}:\left\vert z-(2n\pi)^{2}\right\vert =2n\right\}
$ contains inside only two eigenvalues (counting multiplicities) denoted by
$\lambda_{n}(t)$ and $\lambda_{-n}(t)$ of the operators $L_{t}$ for $|t|\leq
h$. Moreover, $C(n)$ lies in the resolvent set of $L_{t}$ for $|t|\leq h.$
Consider the total projections
\begin{equation}
T_{n}(x,t)=\int_{C(n)}A(x,\lambda,t)d\lambda,
\end{equation}
where
\begin{equation}
A(x,\lambda,t)=\int\limits_{0}^{1}G(x,\xi,\lambda,t)f_{t}(\xi)d\xi
\end{equation}
and $G(x,\xi,\lambda,t)$ is the Green function of the operator $L_{t}.$ It is
well-known that the Green function $G(x,\xi,\lambda,t)$ of $L_{t}$ is defined
by formulas (see [10] pages 36 and 37)%
\begin{equation}
G(x,\xi,\lambda,t)=\frac{H(x,\xi,\lambda,t)}{\Delta(\lambda,t)},
\end{equation}
where
\begin{equation}
H(x,\xi,\lambda,t)=\left\vert
\begin{array}
[c]{ccc}%
\theta(x,\lambda) & \varphi(x,\lambda) & g(x,\xi)\\
\theta-e^{it} & \varphi & g(1,\xi)-e^{it}g(0,\xi)\\
\theta^{\prime} & \varphi^{\prime}-e^{it} & g^{^{\prime}}(1,\xi)-e^{it}%
g^{^{\prime}}(0,\xi)
\end{array}
\right\vert ,
\end{equation}%
\begin{equation}
g(x,\xi)=\pm\frac{1}{2}%
\begin{vmatrix}
\theta(x,\lambda) & \varphi(x,\lambda)\\
\theta(\xi,\lambda) & \varphi(\xi,\lambda)
\end{vmatrix}
\end{equation}
and $\Delta(\lambda,t)$ is defined in (3). In (116) the positive sign being
taken if $x>\xi,$ and the negative sign if $x<\xi.$

Since $\Delta(\lambda,t)$ is continuous in the compact $C(n)\times U(h),$
where $U(h)=\left\{  t\in\mathbb{C}:\left\vert t\right\vert \leq h\right\}  $,
there exists a positive constant $c_{7}$ such that
\begin{equation}
\left\vert \Delta(\lambda,t)\right\vert \geq c_{7},\text{ }\forall
(\lambda,t)\in C(n)\times U(h).
\end{equation}
Therefore using (112)-(116) and taking into account that $f_{t}(x)$ is the sum
of finite number of summands (see (40)), we obtain that for any $x\in
\lbrack0,1]$ the function $T_{n}(x,t)$ is analytic in $U(h)$ and there exist
$c_{8}$ such that
\begin{equation}
\left\vert T_{n}(x,t)\right\vert \leq c_{8}%
\end{equation}
for all $(x,t)\in\lbrack0,1]\times U(h)$. It implies that
\begin{equation}
\int\limits_{\gamma(0,h)}T_{n}(x,t)dt=\int\limits_{[-h,h]}T_{n}(x,t)dt.
\end{equation}
On the other hand, inside of the circle $C(n)$ the operator $L_{t}$ for $t\in
U(h)\backslash(A_{n}\cup A_{-n})$ has $2$ simple eigenvalues $\lambda_{n}(t)$
and $\lambda_{-n}(t),$ where $A_{n}\cup A_{-n}$ is a finite set (see (15)).
Therefore
\begin{equation}
T_{n}(x,t)=a_{n}(t)\Psi_{n,t}+a_{-n}(t)\Psi_{-n,t},\text{ }\forall t\in
U(h)\backslash(A_{n}\cup A_{-n}).
\end{equation}
Besides, by (76), for $t\in\gamma(0,h)$ the eigenvalues are simple and hence
$(f_{t},X_{k,t})\Psi_{k,t}(x)$ is continuous function on $\gamma(0,h)$ for
each $x$ which implies that
\begin{equation}
\int\limits_{\gamma(0,h)}a_{n}(t)\Psi_{n,t}+a_{-n}(t)\Psi_{-n,t}%
dt=\int\limits_{\gamma(0,h)}a_{n}(t)\Psi_{n,t}dt+\int\limits_{\gamma
(0,h)}a_{-n}(t)\Psi_{-n,t}dt.
\end{equation}
Thus it follows from (119)-(121) that%
\begin{equation}
\int\limits_{\gamma(0,h)}a_{n}(t)\Psi_{n,t}dt+\int\limits_{\gamma(0,h)}%
a_{-n}(t)\Psi_{-n,t}dt=\int\limits_{[-h,h]}\left(  a_{n}(t)\Psi_{n,t}%
+a_{-n}(t)\Psi_{-n,t}\right)  dt.
\end{equation}
Now one can readily see that the equalities (110) and (111) follows from
(118), (120) and Proposition 5(b) respectively.

It is clear that there exists a closed curve $\Gamma(0)$ such that the curve
$\Gamma(0)$ lies in the resolvent set of the operator $L_{t}$ for $|t|\leq h$
and all eigenvalues of $L_{t}$ for $|t|\leq h$ that do not lie in $C(n)$ for
$n>N(h)$ belong to the set enclosed by $\Gamma(0).$ Therefore instead of
$C(n)$ using $\Gamma(0)$ and repeating the above arguments we obtain that the
function
\begin{equation}
S_{N}(x,t)=:\sum_{\left\vert n\right\vert \leq N(h)}a_{n}(t)\Psi_{n,t}(x)
\end{equation}
is analytic in $U(h)$ and there exist $c_{9}$ such that
\begin{equation}
\left\vert S_{N}(x,t)\right\vert \leq c_{9}%
\end{equation}
for all $(x,t)\in\lbrack0,1]\times U(h)$ and
\[
\int\limits_{\gamma(0,h)}\sum_{\left\vert n\right\vert \leq N_{h}(0)}%
a_{n}(t)\Psi_{n,t}(x)=\int\limits_{[-h,h]}\left(  \sum_{\left\vert
n\right\vert \leq N_{h}(0)}a_{n}(t)\Psi_{n,t}(x)\right)
\]
Now using (106) and taking into account that the integrals of $S_{N}(x,t)$
over $[-\delta,\delta]$ tend to zero as $\delta\rightarrow0$ (see (124)) we
get the proof of the theorem.
\end{proof}

In the same way we obtain.

\begin{theorem}
Let $f$ be continuous and compactly supported function, $0<h<\frac{1}{15\pi}$
and (76) holds. Then the following equalities hold
\begin{equation}
\int\limits_{\gamma(\pi,h)}f_{t}(x)dt=\int\limits_{[\pi-h,\pi+h]}%
\sum_{n=-N(h)-1}^{N(h)}a_{n}(t)\Psi_{n,t}(x)dt+
\end{equation}%
\[
\sum_{n>N(h)}\int\limits_{[\pi-h,\pi+h]}\left(  a_{n}(t)\Psi_{n,t}%
(x)+a_{-(n+1)}(t)\Psi_{-(n+1),t}(x)\right)  dt,
\]%
\[
\int\limits_{\lbrack\pi-h,\pi+h]}\left(  \sum_{n=-N(h)-1}^{N(h)}a_{n}%
(t)\Psi_{n,t}(x)\right)  dt=\lim_{\delta\rightarrow0}\left(  \sum
_{n=-N(h)-1}^{N(h)}\int\limits_{\delta<\left\vert \pi-t\right\vert \leq
h}a_{n}(t)\Psi_{n,t}(x)dt\right)  ,
\]%
\[
\int\limits_{\lbrack\pi-h,\pi+h]}\sum_{k=n,-(n+1)}a_{k}(t)\Psi_{k,t}%
dt=\lim_{\delta\rightarrow0}\left(  \sum_{k=n,-(n+1)}\int\limits_{\delta
<\left\vert \pi-t\right\vert \leq h}a_{k}(t)\Psi_{k,t}dt\right)  .
\]
Moreover, if $\lambda_{n}(\pi)$ is not an ESS then
\[
\int\limits_{\lbrack\pi-h,\pi+h]}\sum_{k=n,-(n+1)}a_{k}(t)\Psi_{k,t}%
dt=\sum_{k=n,-(n+1)}\int\limits_{[\pi-h,\pi+h]}a_{k}(t)\Psi_{k,t}dt.
\]
The series in (125) converge in the norm of $L_{2}(a,b)$ for every
$a,b\in\mathbb{R}.$
\end{theorem}

Thus by (80) and theorems 3, 5, 6 we have the following spectral expansion theorem

\begin{theorem}
For each continuous and compactly supported function $f$ the spectral
expansion given by the equalities (80) and (100), (108), (125) holds.
\end{theorem}

In the Conclusion 1 we discuss in detail the necessity of the parenthesis (the
handling of the terms $a_{n}(t)\Psi_{n,t}(x)$ and $a_{-n}(t)\Psi_{-n,t}(x))$
in the second row of \ (108) and the convergence of the series with
parenthesis. Now in the following remark we discuss the parenthesis in the
first row of (108).

\begin{remark}
\textbf{ On the parenthesis in (108). }We say that the set%
\begin{equation}
\left\{  a_{k}(t)\Psi_{k,t}(x):k\in\mathbb{T}(\Lambda)\right\}  ,
\end{equation}
where $\mathbb{T}(\Lambda)$ is defined in (59), is a bundle corresponding to
the multiple eigenvalue $\Lambda.$ If $\Lambda$ is not ESS of the operator $L$
then it follows \ from Definition 3 and Remark 1 that all elements
$a_{k}(t)\Psi_{k,t}(x)$ of the bundle (126) are integrable functions on
$[-\varepsilon,0)\cup(0,\varepsilon]$ for all $x$ and for some $\varepsilon.$
If $\Lambda$ is an ESS of the operator $L$ then for some values of
$k\in\mathbb{T}(\Lambda)$ the function $a_{k}(t)\Psi_{k,t}(x)$ for almost all
$x$ is nonintegrable on $[-\varepsilon,0)\cup(0,\varepsilon],$ while some of
elements of the bundle (126) may be integrable. Instead of $C(n)$ using a
small circle enclosing $\Lambda$ and repeating the proof of (118) we see that
the total sum of elements of (126) is bounded due to the cancellations of the
nonintegrable terms of (126). At least two element of the bundle must be
nonintegrable in order to do the cancellations. In fact, we may and must to
huddle together only the nonintegrable elements of the bundle (126). In case
$\Lambda=\lambda_{n}(0)$ and $n\gg1$ the bundle (126) consist of $a_{n}%
(t)\Psi_{n,t}$ and $a_{-n}(t)\Psi_{-n,t}$ and both of then are nonintegrable.
That is why we must to handle they together.

Let $\lambda_{n_{j}}(0)$ for $j=1,2,...,s$ be ESS, where $\left\vert
n_{j}\right\vert \leq N(h).$ Then the set $\left\{  n\in\mathbb{Z}:\text{
}\left\vert n\right\vert \leq N(h)\right\}  $ can be divided into subsets
$\mathbb{T}(\lambda_{n_{j}}(0))$ for $j=1,2,...,s$\ and
\[
\mathbb{K}=\left\{  n\in\mathbb{Z}:\text{ }\left\vert n\right\vert \leq
N(h)\right\}  \backslash\bigcup\limits_{j=1,2,...,s}\mathbb{T}(\lambda_{n_{j}%
}(0)).
\]
Therefore the summations over $\left\{  n\in\mathbb{Z}:\text{ }\left\vert
n\right\vert \leq N(h)\right\}  $ in (108) and (109) can be written as the sum
of summations over $\mathbb{T}(\lambda_{n_{1}}(0)),\mathbb{T}(\lambda_{n_{2}%
}(0)),...,\mathbb{T}(\lambda_{n_{s}}(0))$ and $\mathbb{K}.$ In Theorem 5 to
avoid the complicated notations the summations over $\left\{  n\in
\mathbb{Z}:\text{ }\left\vert n\right\vert \leq N(h)\right\}  $ is taken. We
have the same situation with Theorem 6.
\end{remark}

Now to write the spectral expansion theorems in a compact form we introduce
some notations and definition. For this we parameterize the Bloch eigenvalues
$\lambda_{n}(t)$ and Bloch functions $\Psi_{n,t}(x)$ by quasimomentum $t$
changing in all $\mathbb{R}.$

\begin{notation}
Define $\lambda:\mathbb{R}\rightarrow\mathbb{C}$ by $\lambda(t)=\lambda
_{n}(t-2\pi n)$ for $t\in(2\pi n-h,2\pi(n+1)-h],$ where $n\in\mathbb{Z}.$
Similarly, let $\Psi(x,t)$ and $\Psi^{\ast}(x,t)$, denotes respectively
$\Psi_{n,t-2\pi n}(x)$ and $\Psi_{n,t-2\pi n}^{\ast}(x)$ if $t\in(2\pi
n-h,2\pi(n+1)-h].$ Let $\alpha(t)=(\Psi(\cdot,t),\Psi^{\ast}(\cdot,t))$ and
$a(t)=(f,\Psi(\cdot,t))_{\mathbb{R}}$
\end{notation}

\begin{definition}
A quasimomentum $t$ is said to be singular quasimomentum if $\lambda(t)$ is
ESS. By Theorem 1 the set of singular quasimomenta is the subset of $\left\{
\pi n:n\in\mathbb{Z}\right\}  $. Therefore the definition of the singular
quasimomenta can also be given as follows: $\pi n$ is called a singular
quasimomentum if $\lambda(\pi n)$ is ESS.
\end{definition}

Let $\Lambda=\lambda_{n}(0)$ be ESS. It means that:

Case 1. If $\left\vert n\right\vert >N(h)$ then $\lambda_{n}(0)=\lambda
_{-n}(0).$

Case 2. If $\left\vert n\right\vert \leq N(h)$ then $\lambda_{j}(0)=\Lambda$
for all $j\in\mathbb{T}(\Lambda).$

Then in Case 1 the quasumomenta $\pm2\pi n,$ and in Case 2 the quasimomenta
$2\pi j$ for $j\in\mathbb{T}(\Lambda)$ are the singular quasimomenta
corresponding to the ESS $\lambda_{n}(0)$. In the same way we define the
singular quasimomenta corresponding to the ESS $\lambda_{n}(\pi)$.

As we noted in Remark 2, if $\lambda_{n}(0)$ for $\left\vert n\right\vert
>N(h)$ is ESS then both $a_{n}(t)\Psi_{n,t}$ and $a_{-n}(t)\Psi_{-n,t}$ are
nonintegrable in neighborhoods of $0$ and we must to handle they together. In
the language of Notation 1, it means that if $\lambda(2\pi n)$ for $\left\vert
n\right\vert >N(h)$ is ESS then $a(t)\Psi(x,t)$ is nonintegrable in the
neighborhoods of the singular quasimomenta $2\pi n$ and $-2\pi n$
corresponding to the ESS $\lambda(2\pi n).$ That is why, the handling
$a_{n}(t)\Psi_{n,t}$ and $a_{-n}(t)\Psi_{-n,t}$ in (108) now corresponds to
the handling of the neighborhoods of $2\pi n$ and $-2\pi n$ together.
Therefore we divide the set $\mathbb{R}$ of quasimomenta $t$ into two parts:
the set of neighborhoods of singular quasimomenta and the other part of
$\mathbb{R}$. Similarly, we divide the spectrum $\sigma(L)$ into two part: the
set of neighborhood of ESS and the other part of $\sigma(L).$ For this
introduce the notations.

\begin{notation}
Let $\left\{  \pi n_{j}:\text{ }j=1,2,...,\right\}  $ be the set of the
singular quasimomenta. By Definition 4, $\lambda(\pi n_{j})$ is an ESS and
\[
\mathbb{E=}\left\{  \lambda(\pi n_{j}):j=1,2,...,\right\}  ,
\]
where $\mathbb{E}$ is the set of ESS. For $\left\vert n_{j}\right\vert >N(h)$
define $B_{j}(h)$ and $B_{j}(h,\delta)$ by
\[
B_{j}(h)=(\pi n_{j}-h,\pi n_{j}+h)\cup(-\pi n_{j}-h,-\pi n_{j}+h),\text{
}B_{j}(h,\delta)=B_{j}(h)\backslash B_{j}(\delta),
\]
\ where $0<h<\frac{1}{15\pi}$ and $0<\delta<h.$ For $\left\vert n_{j}%
\right\vert \leq N(h)$ define $B_{j}(h)$ and $B_{j}(h,\delta)$ by
\[
B_{j}(h)=\cup_{n\in\mathbb{T}_{j}}(\pi n-h,\pi n+h),\text{ }B_{j}%
(h,\delta)=B_{j}(h)\backslash B_{j}(h,\delta),
\]
where $\mathbb{T}_{j}=:\mathbb{T}(\lambda(\pi n_{j}))$ and $\mathbb{T}%
(\Lambda)$ is defined in (59). The set $\lambda(B_{j}(h))$ is the part of the
spectrum $\sigma(L)$ of $L$ lying in the neighborhood of the ESS $\lambda(\pi
n_{j}),$ where $\lambda(C)=:\left\{  \lambda(t):t\in C\right\}  $ for
$C\in\mathbb{R}.$ Finally let
\[
B(h)=\cup_{j}B_{j}(h).
\]

\end{notation}

Using this notation and theorems 3, 5 and 6 we obtain

\begin{theorem}
For each continuous and compactly supported function $f$ the following
expansion holds
\begin{equation}
f(x)=\frac{1}{2\pi}\int\limits_{\mathbb{R}\backslash B(h)}a(\lambda
(t))\Psi(x,\lambda(t))dt+\frac{1}{2\pi}\sum\limits_{j}p.v.\int\limits_{B_{j}%
(h)}a(\lambda(t))\Psi(x,\lambda(t))dt
\end{equation}
where the $p.v.$ integral over $B_{j}(h)$ is the limit as $\delta\rightarrow0$
of the integral over $B_{j}(h,\delta).$ The first integral and the series in
(127) converge in the norm of $L_{2}(a,b)$ for every $a,b\in\mathbb{R}.$
\end{theorem}

Now changing the variable to $\lambda$ in (127) as was done in (37) and using
Notation 2 we obtain the following spectral expansion.

\begin{theorem}
For each continuous and compactly supported function $f$ the following
spectral expansion holds
\begin{equation}
f(x)=\frac{1}{2\pi}\int\limits_{\sigma(L)\backslash\lambda(B(h)),}(\Phi
_{+}(x,\lambda)F_{-}(\lambda,f)+\Phi_{-}(x,\lambda)F_{+}(\lambda,f))\frac
{1}{\varphi p(\lambda)}d\lambda+
\end{equation}%
\[
\frac{1}{2\pi}\sum\limits_{j}p.v.\int\limits_{\lambda(B_{j}(h)),}(\Phi
_{+}(x,\lambda)F_{-}(\lambda,f)+\Phi_{-}(x,\lambda)F_{+}(\lambda,f))\frac
{1}{\varphi p(\lambda)}d\lambda
\]
where the $p.v.$ integral over $\lambda(B_{j}(h))$ is the limit as
$\delta\rightarrow0$ of the integral over $\lambda(B_{j}(h,\delta)),$ the
functions $\Phi_{\pm}(x,\lambda)$ and $F_{\pm}(\lambda,f)$ are defined in (34)
and (35). The first integral and the series in (128) converge in the norm of
$L_{2}(a,b)$ for every $a,b\in\mathbb{R}.$
\end{theorem}

Now let us do some conclusion about the obtained spectral expansions.

\begin{conclusion}
At first glance it seems that the obtained spectral expansions have a
complicated form, since the series (108) and (125) converge with parenthesis
(see Theorem 5 and Theorem 6) and in (127) and (128) the $p.v.$ integrals are
used. However, it is only connected with a complicated picture of the spectrum
and projections and nature of the Hill operator with complex periodic
potential. To confirm it, we now explain the necessity of the parenthesis and
$p.v.$ integrals and try to show that the all factors that effect to the
spectral expansion are taken into account. First, note that it follows from
the Notation 2 that the integrals in (127) and (128) are taking over all
$\mathbb{R}$ and $\sigma(L)$ except the discrete sets
\[
\left\{  \pi n_{j}:j=1,2,...,\right\}  \text{ }\And\text{\ }\mathbb{E=}%
\left\{  \lambda(\pi n_{j}):j=1,2,...,\right\}
\]
respectively. Since the corresponding integrals about the points of those sets
do not exist, we use the $p.v.$ integral, that is, the limit as $\delta
\rightarrow0.$ Moreover, the sets $B_{j}(h)$ are constructed in the way which
takes into account the requisite parenthesis in (108) and (125). Let us
explain, in detail, why the parenthesis and limits as $\delta\rightarrow0$ are
necessary for the spectral expansion for the general complex-valued periodic potentials:

\textbf{Necessity of the parenthesis in (108) and (125) and }$\mathbf{p.v.}%
$\textbf{ integrals in (127) and (128). }The series in (108) and (125)
converge with parenthesis and in parenthesis is included only the integrals of
the functions corresponding to splitting eigenvalues. The parenthesis is
necessary, due to the following. If $n\gg1$ and $\lambda_{n}(0)$ is ESS, then
$\lambda_{n}(0)$ is a double eigenvalue, $\lambda_{n}(0)=\lambda_{-n}(0)$ and
both of the functions $a_{n}(t)\Psi_{n,t}$ and $a_{-n}(t)\Psi_{-n,t}$ has
nonintegrable singularities (see (72) and the definition of $a_{n}(t)$ in
(44)), that is, their integrals do not exist. However, the integral
\begin{equation}
\int\limits_{\lbrack-h,h]}\left(  a_{n}(t)\Psi_{n,t}+a_{-n}(t)\Psi
_{-n,t}\right)  dt
\end{equation}
exists. Moreover, even if $\lambda_{n}(0)$ and $\lambda_{n}(\pi)$ are not ESS
respectively, then it is possible that the norm of
\begin{equation}
\int\limits_{\lbrack-h,h]}a_{n}(t)\Psi_{n,t}(x)dt\text{ }\And\int
\limits_{\lbrack\pi-h,\pi+h]}a_{n}(t)\Psi_{n,t}(x)dt\text{ }%
\end{equation}
do not tend to zero as $n\rightarrow\infty.$ Therefore the series in (108) and
(125 ) do not converge without parenthesis. More precisely, the series (108)
and (125 ) converge without parenthesis if and only if there exist $h>0$ such
that the the first and second integrals in (130) respectively exist and tend
to zero as $n\rightarrow\pm\infty.$ (see Theorem 10). Note that this situation
agree with the well-known result [13] that the root functions of the operators
generated by a ordinary differential expression in $[0,1]$ with regular
boundary conditions, in general, form a Riesz basis with parenthesis and in
parenthesis should be included only the functions corresponding to the
splitting eigenvalues. In particular, the periodic $(t=0)$ and antiperiodic
$(t=\pi)$ boundary conditions require the parenthesis. It is natural that in
the case of the operator $L$ generated by a ordinary differential expression
in $(-\infty,\infty)$ we included in parenthesis the Bloch functions
$\Psi_{n,t}(x)$ near two $t=0$ (see (108)) and $t=\pi$ (see (125)).

The using of the $p.v.$ integral about singular quasumomenta and ESS in (127)
and (128) respectively is necessary, since the integrals about those points do
not exist. We do not need the $p.v.$ integral if and only if the operator $L$
has no ESS.

Thus in the general case we should use the parenthesis and $p.v.$ integrals
and one can obtain a spectral expansion without parenthesis and $p.v.$
integrals if and only if $L(q)$ has no ESS and the integrals in (130) tend to
zero as $\left\vert n\right\vert \rightarrow\infty.$ Namely, we have the following.
\end{conclusion}

\begin{theorem}
For each continuous and compactly supported function $f$ we have the following
spectral decompositions
\[
f(x)=\frac{1}{2\pi}\sum_{k\in\mathbb{Z}}\int_{0}^{2\pi}a_{k}(t)\Psi
_{k,t}(x)dt=\frac{1}{2\pi}\int\limits_{\sigma(L)}(\Phi_{+}(x,\lambda
)F_{-}(\lambda,f)+\Phi_{-}(x,\lambda)F_{+}(\lambda,f))\frac{1}{\varphi
p(\lambda)}d\lambda
\]
if and only if $L(q)$ has no ESS and there exists $h>0$ such that the
integrals in (130) tend to zero as $\left\vert n\right\vert \rightarrow\infty$.
\end{theorem}

\end{document}